\documentclass[11pt,letterpaper,reqno]{amsart}
\usepackage{vmargin,color}
\usepackage[latin1]{inputenc}
\usepackage{amsmath,amsfonts,amsthm,epsfig,graphicx,amssymb}
\usepackage{caption}

\usepackage{mathtools}   
\mathtoolsset{showonlyrefs}
\usepackage{hyperref}

\def\F{\mathcal{F}}

\def\H{\mathcal{H}}

\def\p{\mathbf{p}}

\def\N{\mathbb N}

\def\R{\mathbb R}

\def\Om{\Omega}
\def\Si{\Sigma}

\def\a{\alpha}

\def\e{\varepsilon}
\def\k{\kappa}
\def\l{\lambda}
\def\s{\sigma}

\def\vphi{\varphi}

\def\Id{{\rm Id}\,}
\def\dist{{\rm dist}}

\def\diam{{\rm diam}}

\def\id{{\rm id}}
\def\cof{{\rm cof}}
\def\dom{{\rm dom}}

\def\weak{\rightharpoonup}
\def\weakstar{\stackrel{*}{\rightharpoonup}}

\def\ov{\overline}
\def\pa{\partial}

\def\bd{{\rm bd}}
\def\INT{{\rm int}\,}

\def\Div{{\rm div}}
\def\hd{{\rm hd}}
\DeclareMathOperator*{\supp}{{\rm supp}}

\newcommand{\overbar}[1]{\mkern1.5mu\overline{\mkern-1.5mu#1\mkern-1.5mu}\mkern 1.5mu}

\newtheorem{theorem}{Theorem}[section]
\newtheorem{remark}[theorem]{Remark}
\newtheorem{definition}[theorem]{Definition}

\newtheorem{proposition}[theorem]{Proposition}
\newtheorem{lemma}[theorem]{Lemma}

\newtheorem{manualtheoreminner}{Theorem}
\newenvironment{manualtheorem}[1]{%
  \renewcommand\themanualtheoreminner{#1}%
  \manualtheoreminner
}{\endmanualtheoreminner}

\numberwithin{equation}{section}
\numberwithin{figure}{section}

\begin{document}

\title{A Heintze-Karcher inequality with free boundaries and applications to capillarity theory}

\author{Matias G. Delgadino}
\address{
Department of Mathematics, The University of Texas at Austin, 2515 Speedway, Austin, TX 78712-1202}
\email{matias.delgadino@math.utexas.edu}
\author{Daniel Weser}
\address{Department of Mathematics, The University of North Carolina at Chapel Hill, 120 E Cameron Avenue,
Chapel Hill, NC 27599-3250}
\email{weser@unc.edu}
%

\maketitle

\begin{abstract}
  {\rm In this paper we analyze the shape of a droplet inside a smooth container. To characterize their shape in the capillarity regime, we obtain a new form of the Heintze--Karcher inequality for mean convex hypersurfaces with boundary lying on curved substrates.}\\
  \textbf{MSC 2020: 53C24, 35J25, 53C21}\\
  \textbf{Keywords:} Heintze-Karcher's inequality, capillary hypersurface, CMC hyper-
surface, Alexandrov's theorem.
\end{abstract}


\section{Introduction} \subsection{The Heintze-Karcher inequality and Alexandrov's theorem} Given a bounded open connected set $\Om\subseteq\R^{n+1}$ with smooth boundary $\pa\Om$ and positive scalar mean curvature $H_\Om$, the {\it Heintze-Karcher inequality} states that
\begin{equation}
  \label{hk inq}
  (n+1)|\Om|\le\int_{\partial \Omega}\frac{n}{H_\Om}\,,
\end{equation}
with equality if and only if $\Om$ is a Euclidean ball. Here $|\Om|$ denotes the volume of $\Om$, while $H_\Om$ is computed with respect to the outer unit normal $\nu_\Om$ to $\Om$. 

The Heintze-Karcher inequality is closely related to Alexandrov's theorem, which states that if $H_\Om$ is constant (and $\Om$ is bounded and connected), then $\Om$ is a Euclidean ball. Indeed, if $H_\Om$ is constant, then by applying the divergence theorem twice (once in $\Om$ and once tangentially to $\pa\Om$) to the identity vector field one finds
\begin{eqnarray}\nonumber
(n+1)|\Om|=\int_\Om\Div(x)&=&\int_{\pa\Om}(x\cdot\nu_\Om)=\frac1{H_{\pa\Om}}\int_{\pa\Om}(x\cdot\nu_\Om)H_{\pa\Om}
\\\label{double div thm}
&=&\frac1{H_{\pa\Om}}\int_{\pa\Om}\Div^{\pa\Om}(x)=\frac{n\,\H^n(\pa\Om)}{H_{\pa\Om}}
=\int_{\pa\Om}\frac{n}{H_{\pa\Om}}\,,
\end{eqnarray}
where $\H^n$ denotes the $n$-dimensional Hausdorff measure in $\R^{n+1}$ and $\Div^{\pa\Om}$ the tangential divergence with respect to $\pa\Om$. Thus the relation between \eqref{hk inq} and Alexandrov's theorem is that one can prove the latter by characterzing the equality cases in \eqref{hk inq}.

This interesting connection has motivated many authors to investigate the Heintze-Karcher inequality (and variants of it) for proving rigidity theorems. Two proofs of the Heintze-Karcher inequality that are particularly relevant for the present study are the one by Montiel and Ros \cite{montielros} (see Lemma \ref{lemma montiel ros argument} and Remark \ref{remark montielros} below; this proof is also related to the argument recently used by Brendle \cite{brendle} to extend Alexandrov theorem to a large class of warped product manifolds) and the one by Ros \cite{ros87ibero}. Ros' argument is based on the study of the torsion potential $u$ of $\Om$, namely, the solution to the problem
\begin{equation}
  \left\{\begin{split}
    -\Delta u=1&\qquad\mbox{in $\Om$}\,,
    \\
    u=0&\qquad\mbox{on $\pa\Om$}\,.
  \end{split}\right .
\end{equation}
As a by-product of this proof, one sees that when equality holds in \eqref{hk inq}, then $-\nabla^2u=(n+1)^{-1}\,\Id$ in $\Om$ and there exists $\l>0$ such that $|\nabla u|=-\nabla u\cdot\nu_\Om=\l$ on $\pa\Om$. In particular, the connected components of $\Om$ are balls of the same radius, which can be computed in terms of $\l$ by exploiting the Dirichlet condition on $u$.

The analysis of this inequality found an application in a basic question of capillarity theory, namely, the geometric descriptions of critical points of the Gauss energy of a liquid droplet in the capillarity regime. Indeed, using the Heintze-Karcher inequality, the pioneering work of \cite{ciraolomaggi} showed that if a smooth set has mean curvature uniformly close to a constant, then the set is almost a finite union of tangent balls of equal radius. In \cite{delgadino2019alexandrov}, Alexandrov's theorem was revisited for non smooth sets to account for bubbling phenomenon. More specifically, sets of finite perimeter with distributional mean curvature equal to a constant are shown to be the union of disjoint and tangent balls of equal radius. This result was later quantified for smooth sets in \cite{julin2020quantitative} with the closeness of the mean curvature to a constant measured in $L^n(\pa\Om)$, which allowed the authors to show exponential convergence of the volume preserving mean curvature flow to a union of balls in $\R^3$. For the anisotropic generalization of some of these results, see \cite{delgadino2018bubbling,de2020uniqueness}.

\subsection{The Heintze-Karcher inequality on an ideal substrate}
Motivated by the study of droplets on a substrate, we look for a natural extension of the Heintze-Karcher inequality from the case of smooth sets in $\R^{n+1}$ to that of subsets $\Om$ of the half-space $H=\{x_{n+1}>0\}$. As depicted in Figure \ref{fig generalnotation},
\begin{figure}
    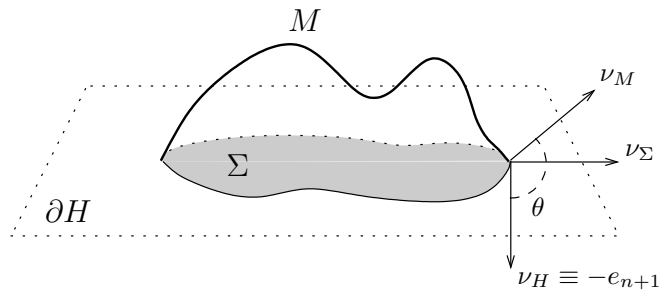\caption{\small{The setting for the Heintze-Karcher inequality on an ideal substrate.}}\label{fig generalnotation}
\end{figure}  
we consider sets $\Om\subseteq \{x_{n+1}>0\}$ such that $M=\ov{\pa\Om\cap\{x_{n+1}>0\}}$ is a smooth hypersurface with boundary, whose set of boundary points $\bd(M)$ satisfies $\bd(M)=M\cap\{x_{n+1}=0\}$. We endow $M$ with the orientation $\nu_M$ defined by $\nu_\Om$, and, after defining for each $x\in\bd(M)$ the angle $\theta(x)$ through
\begin{equation}
  \label{nondegeneracy}
\nu_M(x)\cdot \nu_H=\cos\theta(x) \,,
\end{equation}
we assume the non-degeneracy condition (related to \eqref{youngs law}) that  $\theta(x)\in(0,\pi)$ for every $x\in\bd(M)$. In particular, the wetted region $\Si=\pa\Om\cap\{x_{n+1}=0\}$ is a smooth set with $\bd(\Si)=\bd(M)$, and one can decompose
\[
\nu_M(x)=\cos\theta(x)\nu_H+\sin\theta(x)\,\nu_\Si(x)\qquad\forall x\in\bd(M)\,,
\]
where $\nu_\Si$ denotes the outer unit normal to $\Si$ in $\{x_{n+1}=0\}$.

A basic difficulty here is understanding what form the new Heintze-Karcher inequality should take; indeed, the right-hand side of \eqref{hk inq} makes no sense, as $H_\Om$ becomes a measure on $\bd(M)$. The answer is found by looking at the Montiel-Ros proof \cite{montielros}. Exploited in this setting, their argument rather naturally suggests the discussion of the problem under the additional assumption that
\[
\mbox{either $\theta(x)\in(0,\pi/2]$ or $\theta(x)\in[\pi/2,\pi)$ for every $x\in\bd(M)$.}
\]
Although this may look a bit artificial, this kind of restriction is actually natural from the physical point of view, as it amounts to consider the case of {\it hydrophobic} ($0<\theta\le\pi/2$) or {\it hydrophilic} ($\pi/2\le\theta<\pi$) substrates. It is perhaps not surprising, then, that one finds {\it two different generalizations} of the Heintze-Karcher inequality \eqref{hk inq}, which we collectively call the {\it Heintze-Karcher inequalities on an ideal substrate}:

\begin{figure}
  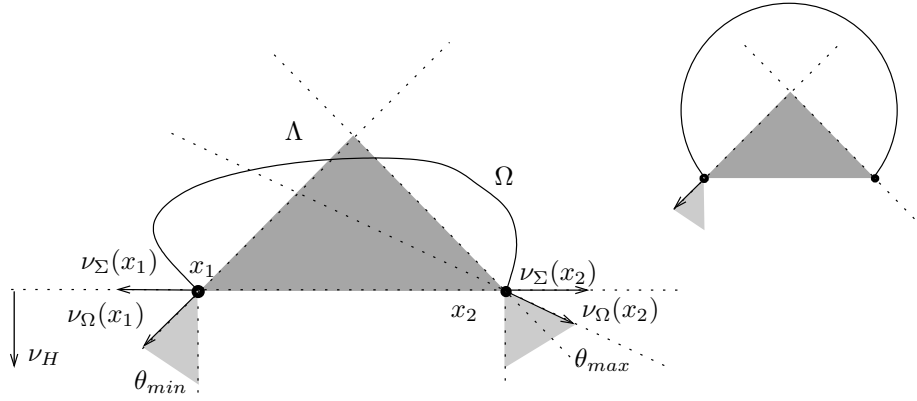\caption{\small{In the hydrophobic case, the ``pyramid'' $\Lambda$ is contained in $H$. Notice that $\Lambda$ may not be contained in $\Om$. When $\Om$ is a ball which intersects $\pa H$ below its equator (picture on the top right corner), then $\Lambda$ is the cone with vertex at the center of the ball, spanned over the wetted region.}}\label{fig lambdahydrophobic}
\end{figure} 
\begin{figure}
  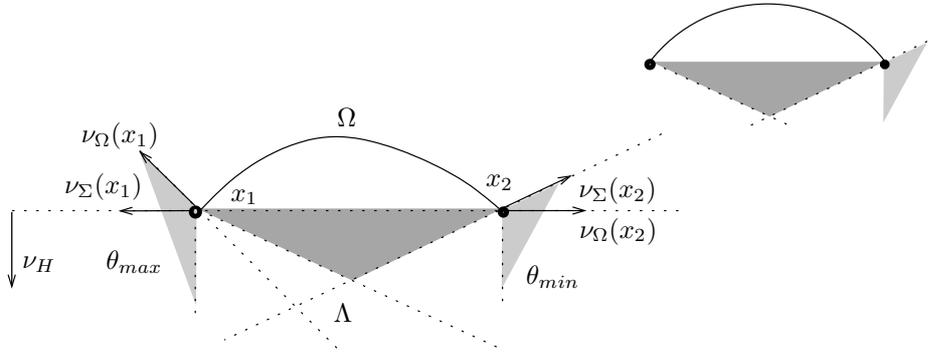\caption{\small{In the hydrophilic case, $\Lambda$ is contained in $\{x_{n+1}<0\}$, and thus it is automatically disjoint from $\Om$. In this regime the optimal case is when $\Om$ is a ball intersecting $\pa H$ above its equator. As in the hydrophobic case, $\Lambda$ can then be described as the cone with vertex at the center of the ball, spanned over the wetted region.}}\label{fig lambdahydrophilic}
\end{figure}

\begin{theorem}[Heintze-Karcher inequality on an ideal substrate]\label{thm hk sub}
  If $\Om\subseteq H=\{x_{n+1}>0\}$ is a bounded connected open set such that $M=\ov{H\cap\pa\Om}$ is a hypersurface with boundary satisfying the hypotheses \textnormal{{\bf(h1)}--{\bf(h4)}} from Section~\ref{sec:hyp}, then
  \begin{equation}\label{hk inq sub proof}
  \begin{split}
    &(n+1)\big(|\Om|-|\Lambda|\big)\le\int_{M}\frac{n}{H_M}\qquad\mbox{in the hydrophobic case}\,,
    \\
    &(n+1)\big(|\Om|+|\Lambda|\big)\le\int_{M}\frac{n}{H_M}\qquad\mbox{in the hydrophilic case}\,,
  \end{split}
  \end{equation}
  where in the hydrophobic case
  \begin{equation}
    \label{lambda hydrophobic}
      \Lambda=\Big\{(z,t)\in \Sigma\times [0,\infty): 0\le t < d(z,\pa\Sigma)\tan\theta( P_{\partial\Sigma}(z))\Big\}\,,
  \end{equation}
  and in the hydrophilic case  
    \begin{equation}
      \label{lambda hydrophilic}
      \Lambda=\Big\{(z,t)\in \Sigma\times (-\infty,0): d(z,\pa\Sigma)\tan\theta (P_{\partial\Sigma}(z))\le t \le 0 \Big\}\,,
  \end{equation}
  where $P_{\partial\Sigma}:\Sigma\to \pa\Sigma$ is the projection onto the boundary.
  
  In both cases, equality holds in \eqref{hk inq sub proof} if and only if $\Om=B_r(t\,e_{n+1})\cap H$ for some $r>0$ and $|t|<r$.
\end{theorem}
The sets $\Lambda$ are depicted in Figures \ref{fig lambdahydrophobic} and \ref{fig lambdahydrophilic}, which illustrate the hydrophobic and hydrophilic cases, respectively. The Heintze-Karcher inequalities on a substrate are both sharp, and only a ball meeting $\pa H$ orthogonally ($\theta\equiv\pi/2$) is an equality case for {\it both} inequalities.

Comparing Theorem \ref{thm hk sub} to other recent results in this direction, we first mention the work of Jia, Xia and Zhang \cite{jia2023heintze}, in which the authors prove an alternative Heintze-Karcher inequality for sets over an ideal substrate or contained inside a half-ball and for fixed angles. A strengthened quantitative version of their result for slightly curved substrates and variable angles can be found in Proposition~\ref{lem:deficitconvergence2}, which is used for our compactness result. In follow up work, the above authors together with G. Wang, lifted this restriction to variable angles, see \cite[Theorem 1.1]{jia2022heintze}. In comparison with Theorem~\ref{thm hk sub}, \cite[Theorem 1.1]{jia2022heintze} considers the worst case scenario in terms of the angle of the ``pyramids'' we construct, see \ref{fig lambdahydrophilic} and \ref{fig lambdahydrophobic}.

\subsection{Droplets in a container} We now describe the main focus of this paper, the sessile or pendant droplet in capillarity theory. For more details, we refer any interested reader to \cite{Finn} for a classical overview of this problem. We consider the Gauss energy of a droplet of fixed volume $m$ sitting in a container: if $K\subseteq\R^{n+1}$ an open set with $C^2$-boundary denotes the container and the open set $\Om\subseteq K$ denotes the region occupied by the droplet, then the Gauss energy of $\Om$ is given by
\begin{equation}
  \label{gauss energy H}
\F(\Om)=P(\Omega;K)+\int_{\pa\Om\cap\pa K}\s\,d\H^n+\int_{\Om}\,g(x)\,dx\qquad\mbox{with $|\Omega|=m$ fixed.}
\end{equation}
where $m>0$ is a given volume parameter, the function $g:K\to \R$ is the {\it potential energy density}, and $\s:\pa K\to[-1,1]$ is the {\it relative adhesion coefficient}, which measures the relative mismatch between the surface tension of the liquid/air interface $K\cap\pa\Om$ and the liquid/solid interface $\pa\Om\cap\pa K$. Typically, one sets $g(x)=g_0\,\rho\,x_{n+1}$, where $g_0$ is Earth's gravity and $\rho$ the density of on the liquid. Capillarity phenomena are characterized by the dominance of the surface tension energy, that is of $\H^n(K\cap\pa\Om)+\int_{\pa \Om\cap\pa K}\s$ over the potential energy $\int_\Om g$. This is definitely the case when the volume parameter $m$ is small enough in terms of $g$, see \eqref{anisotropic} for the expression of the energy after scaling by the volume. Although for mathematical simplicity the relative adhesion coefficient $\s$ is sometimes assumed to be constant, several relevant physical phenomena such as hysteresis and contact line pinning can be explained by local inhomogeneities of $\s$. Beyond the classical reference \cite{huh1977effects}, we refer the reader to \cite{caffarelli2007capillary1,caffarelli2007capillary} for a modern homogenization approach to rapidly oscillating $\s$.

In this work, we address the description of the critical points of $\F$ when $m$ is small. This problem leads to a question related to the one considered in \cite{ciraolomaggi}. To explain this point, let us say that a bounded set $\Om$ is a critical point of \eqref{gauss energy H}. If $M=\ov{K\cap\pa\Om}$ is a smooth hypersurface with boundary, then the Euler-Lagrange equations are the following:
\begin{equation}\label{el interior}
  H_\Om(x)+g(x)=\l\,\,\qquad\mbox{for every $x\in \INT(M)$}\,,
\end{equation}
\begin{equation}
\label{youngs law}
  \nu_M(x)\cdot\nu_K=\s(x)\qquad\mbox{for every $x\in\bd(M)$}\,.
\end{equation}
where $\bd(M)=M\cap\pa K$ and $\INT(M)=\pa\Omega\cap K$. Condition \eqref{youngs law} is known as {\it Young's law}, and it has been formulated here by taking $\nu_M$ to be the orientation of $M$ obtained by continuously extending $\nu_\Om$ from $\INT(M)$ to $\bd(M)$. Condition \eqref{el interior} is the Euler-Lagrange equation of the perimeter, and $\lambda$ is a Lagrange multiplier that arises from the volume constraint. 

In the ideal case when $g=0$ and $\s\equiv\s_0$, then $M$ is a constant mean curvature hypersurface contained in $K$ whose boundary meets $\pa K$ at a constant angle. {\it Wente's theorem} \cite{wente1980}, which was originally obtained by exploiting the moving planes method, shows that when $\pa K$ is flat then $M$ is a spherical cap. We also mention \cite{HHW20} for the state of the art on the moving planes method with minimal regularity requirements. When $g$ is non-trivial and $|\Omega|$ is small, we are back to a problem similar to the one considered in \cite{ciraolomaggi}, with the added difficulties of having to deal with a surface with boundary and almost constant contact angle.  In \cite{maggi2016shape}, this problem was addressed for the case of global minimizers, showing that as $m\to0^+$ global minimizers converge, up to a re-scaling, to a ball intersected with a half-space. In this work, we extend this result to {\it local} minimizers of \eqref{gauss energy H}, which we define as follows:

\begin{definition}
A set of locally finite perimeter $\Omega\subseteq K$ of volume $|\Omega|=m$ is said to be a \textnormal{volume constrained local minimizer (VCLM)} with diameter $\e_0>0$ of the Gauss capillarity energy if 
 $$
\mathcal{F}_{K,\sigma,g}(\Omega)\le\mathcal{F}_{K,\sigma,g}(\tilde{\Omega})
$$
for every $\tilde{\Omega}\subseteq K$ satisfying $|\tilde{\Omega}|=m$ and $diam(\tilde{\Omega}\triangle \Omega)<\e_0 m^{\frac{1}{n+1}}$.
\end{definition}

\smallskip

Analogous to the global minimizer case, our result is that as $m\to 0^+$, up to re-scaling, local minimizers converge to a single ball intersected with the half-space $\{x_{n+1}>0\}$.

\smallskip

\begin{theorem}\label{thm:locmin}
Consider a fixed compact container $K\subseteq\R^{n+1}$ with $C^2$-boundary and a Lipschitz regular hydrophilic adhesion coefficient $\sigma:\partial K \to (-1,0)$. Given constants $C>1$ and $\e_0>0$, for each $m>0$, we consider $\mathcal{G}_{m,\e_0,C}$ a subset of smooth connected sets, which are volume constrained local minimizers of $\F_{K,\sigma,g}$ with diameter $\e_0$ and volume $m$. Moreover, we assume that any $\Omega_m\in\mathcal{G}_{m,\e_0,C}$ satisfies the uniform non-degeneracy conditions for perimeter, diameter and wetter region:
 \begin{itemize}
    \item $P(\Omega_m;K)\le C m^{\frac{n}{n+1}}$,
    \item $\diam(\Omega_m)\le C m^{\frac{1}{n+1}}$,
    \item $\H^n(\Sigma_m)\ge C^{-1}m^{\frac{n}{n+1}}$\,.
 \end{itemize}
Then, we have that as $m\to0^+$ any element sequence converges to a spherical cap:
\begin{equation}\label{eq:a1}
    \lim_{m\to0^+}\sup_{\Omega_m\in \mathcal{G}_{m,\e_0,C}}\inf_{x_0,Q,\theta_*} m^{-1}|\Omega_m\triangle (m^{\frac{1}{n+1}}(x_0+Q[B_{\theta_*}]))|=0,
\end{equation}
where $B_{\theta_*}$ is the ball which intersects $\{x_{n+1}=0\}$ with angle $\theta_*\in(\pi/2,\pi)$ and has volume $|B_{\theta_*}\cap \{x_{n+1}>0\}|=1$, and $Q[B_{\theta_*}]$ is a rigid motion of $B_{\theta_*}$. 
\end{theorem}

\smallskip

\begin{remark}
  We note that for $n=2$ local minimizers are regular, see \cite{taylor1977boundary,dephilippismaggiCAP-ARMA}.
\end{remark}

\begin{remark}
    The convergence in \eqref{eq:a1} is actually stronger, see \eqref{eq:a2} in Theorem~\ref{mainthm}.
\end{remark}

To obtain this characterization, we prove a compactness result for Wente's theorem. The non-degeneracy conditions in the statement of theorem are the minimal ones necessary to be able to rescale the problem appropriately. The proof of the compactness theorem occupies the bulk of this paper.

\medskip

\subsection{Compactness}

\smallskip


To prove Theorem~\ref{thm:locmin}, we will re-scale the problem so that the re-scaled sets have unit volume. Due to the $C^2$-regularity of the container, the re-scaled boundary will become flat in the limit. Consequently, our compactness theorem takes the following (simplified) form (see Section \ref{sec:compactness} for the precise statement):

\begin{theorem}\label{mainthm}
Let $\{K_l\}_{l\in\N}$ be a sequence of connected open sets with connected $C^2$-boundaries which converge to the half space $\{x_{n+1}>0\}$ in the following sense:
\begin{equation}\label{containers conv to half space}
\|A_l\|_{C^0_{loc}(\pa K_l)}\to 0,\qquad \|e_{n+1}\cdot\nu_{K_l}\|_{C^0_{loc}(\pa K_l)}\to 0\qquad\mbox{and}\qquad \|x\cdot\nu_{K_l}\|_{C^0_{loc}(\pa K_l)}\to 0\,,
\end{equation}
where $A_l$ is the second fundamental form of $\pa K_l$.

Let $\{\Omega_l\}_{l\in\N}$ be a sequence of connected open sets such that their boundaries $\{M_l\}_{l\in\N}$ are smooth (see hypotheses {\bf(h1)}-{\bf(h4)} and notation in Section~\ref{sec:hyp}).
If there exists a positive constant $\lambda_0>0$ such that
    \begin{equation}\label{hyp: conv mc}
        \left\|H_l-\lambda_0\right\|_{C^0(M_l)} \to 0\,,
    \end{equation}
a fixed angle $\theta_0\in(\pi/2,\pi)$ such that
\begin{equation}\label{hyp: conv angle}
    \|\theta_l-\theta_0\|_{C^0(\bd(M_l))}\to 0\,,
  \end{equation}
and a bounded set of finite perimeter $\Omega_\infty\subseteq\{x_{n+1}>0\}$ such that
\begin{equation}\label{hyp: conv in per}
    P(\Omega_l;K_l)\to P(\Omega_\infty;\{x_{n+1}>0\}) \qquad \textnormal{and} \qquad |\Omega_l\Delta\Omega_\infty|\to0,
\end{equation}
then $\Omega_\infty$ is equal to the finite union of disjoint tangent balls of radius $n/\lambda_0$ intersected with $\{x_{n+1}>0\}$ and is such the that any ball in $\Omega_\infty$ which intersects the hyperplane $\{x_{n+1}=0\}$ does so with contact angle $\theta_0$. Moreover, the wetted regions converge in the following sense:
    \begin{equation}\label{eq:a2}
    \lim_{l\to\infty}  |\H^n(\Sigma_l)-\H^n(\Sigma_\infty)|+|\H^{n-1}(\pa\Sigma_l)-\H^{n-1}(\pa\Sigma_\infty )| \;=\; 0 \,.
    \end{equation}
\end{theorem}

\smallskip


The proof of Theorem~\ref{mainthm} follows from deriving a Heintze-Karcher-like inequality (see Proposition~\ref{prop})
\begin{align}
      &\frac{|\Omega|-\gamma\H^{n}(\Sigma)}{n+1} \left( \int_{M} \frac{n}{H_M}d\H^n - (n+1)\big(|\Omega|-\gamma\H^{n}(\Sigma)\big) \right)\\
      &\quad+2(n+1)\e \bigg\{\int_M\frac{d\H^n}{H_M}\left(\frac{|\Omega|}{n+1}+4(d_\Omega^2+\gamma^2)(|\Omega|+\gamma \H^n(\Sigma))\right)+\max\frac{1}{H_M}\left(|\Omega|-\gamma\H^n(\Sigma)\right)^2 \bigg\} \ge 0
      \label{prop:hk} 
\end{align}
by utilizing an appropriately modified torsion potential
\begin{equation}\label{torsion potential}
    \begin{cases}
        \hfill -\Delta u=1&\mbox{ in $\Omega$}\\
        \hfill u=0&\mbox{ on $M$}\\
        \hfill \partial_{\nu_K}u=\gamma&\mbox{ on $\Sigma$}\,,
    \end{cases}
\end{equation}
where
\begin{equation}\label{gamma}
\gamma=-\frac{n}{n+1}\frac{\H^n(\Sigma)}{\int_{\partial\Sigma}\tan\theta(x)\,d\H^{n-1}(x)} \,.
\end{equation}
Although this is the guiding principle of the proof in \cite{ciraolomaggi}, this is where the similarities with our proof end. The mixed boundary conditions in \eqref{torsion potential} made obtaining uniform Lipschitz estimates impossible at the present time. In fact, the need to restrict to the hydrophilic regime $\theta\in(\pi/2,\pi)$ is due to the lack of $C^1$-regularity in the hydrophobic case, see Remark~\ref{rem:c1reg} and Lemma~\ref{torsionC1}. We sidestep these difficulties by utilizing varifold convergence of the boundaries and density estimates to imply Hausdorff convergence of the boundaries. The characterization of the limit set as a union balls follows from showing that the case of equality in the Heintze-Karcher-like inequality \eqref{prop:hk} is obtained in the limit. In fact, we can show the following quantitative inequality:
\begin{align}\label{almost equality}
    &\left|(n+1)(|\Omega|-\gamma\H^n(\Sigma))-\int_M\frac{n}{H_M}\;d\H^n\right|\\
    &\qquad\le C\left(\e+\|\theta-\theta_0\|_{C^0(\pa\Sigma)}+\frac{\|H_M-\lambda\|_{C_0(M)}}{\lambda}\right)\left(d_\Omega+\frac{1}{\lambda}\right)\H^n(M),
\end{align}
where $\gamma$ is defined in \eqref{gamma}.

\medskip

\subsection{Organization of the paper}
Section~\ref{sec:hyp} sets the general notation for the rest of the paper. Section~\ref{sec:tp} contains all the necessary background for the proof of Theorem~\ref{mainthm}, which can be found in Section~\ref{sec:compactness}. Sections~\ref{min} and \ref{sec:hkideal} prove Theorems~\ref{thm:locmin} and \ref{thm hk sub}, respectively.

\subsection*{Acknowledgements}
The authors would like to thank Guido De Philippis and Francesco Maggi for useful discussions. MGD was partially supported by NSF-DMS-2205937.MGD and DW were partially supported by NSF-DMS RTG 1840314.

\section{Hypotheses and Notation}\label{sec:hyp}
Let $n\geq2$, and denote by $K\subseteq \R^{n+1}$ an open set with $C^2$-boundary. We will consider non-empty bounded connected open sets $\Omega\subseteq K$ with finite perimeter,
and we will denote by $\nu_{\Om}$ the measure-theoretic outer unit normal to $\Om$ (which agrees with the classical one if the $\pa\Om$ is of class $C^1$). We assume that
\[
    M=\ov{\pa\Om\cap K}
\]
is a smooth hypersurface with boundary in $\R^{n+1}$. We denote by $\bd(M)$ and by $\INT(M)$ the boundary and the interior points of $M$, respectively, and make the following assumptions:

\medskip

\noindent {\bf (h1):} $\bd(M)$ is non-empty and contained in $\pa K$, that is
  \[
  \bd(M)=M\cap\pa K\qquad \INT(M)=M\cap K=\pa\Om\cap K\,.
  \]
  In particular, $\nu_{\Om}$ is classically defined on $M\cap K=K\cap\pa\Om$, and an orientation $\nu_M$ of $M$ can be defined by setting
  \[
  \begin{split}
    \nu_{M}(x)=\nu_{\Om}(x)&\qquad\forall x\in \INT(M)\,,
    \\
    \nu_{M}(x)=\lim_{\stackrel{y\in M \cap K}{y\to x}}\nu_{\Om}(y)&\qquad\forall x\in\bd(M)\,.
  \end{split}
  \]
  
\medskip

\noindent {\bf (h2):} if we denote by $H_{M}$ the scalar mean curvature of $M$ with respect to $\nu_{M}$, then
  \[
  H_{M}>0\qquad\mbox{on $M$}\,.
  \]
  
  \medskip

\noindent {\bf (h3):} if we define $\theta:\bd(M)\to[0,\pi]$ by setting
  \[
  \cos\theta(y)=\nu_{M}(y)\cdot\nu_{K}\qquad\forall x\in\bd(M)\,,
  \]
  then
  \[
  \theta(x)\in(0,\pi)\qquad\forall x\in\bd(M)\,.
  \]
  In particular, $\pa\Om\cap\pa K$ is a smooth hypersurface in $\R^{n+1}$, $\Sigma:=\INT(\pa\Om\cap\pa K)$ is an open set with smooth boundary in $\pa H=\R^n\times\{0\}$, and
  \[
  \bd(M)=\pa \Sigma\,,
  \]
  where $\pa\Si$ denotes boundary of $\Si$ in $\pa H$, whose outer unit normal is denoted by $\nu_\Sigma$.

  \medskip
  
  \noindent {\bf (h4):} denoting by $\theta_{min}$ and $\theta_{max}$ the extreme values of $\theta$ on $\bd(M)$, we shall assume that
  \begin{equation}
    \begin{split}
      &\mbox{either $[\theta_{min},\theta_{max}]\subseteq (0,\pi/2)$ (hydrophobic substrate)}\,,
      \\
      &\mbox{or \hspace{.5cm} $[\theta_{min},\theta_{max}]\subseteq (\pi/2,\pi)$ (hydrophilic substrate)}\,.
    \end{split}
  \end{equation}
  
  \medskip

  With the above notation in force, for every $x\in\bd(M)$ one has
  \begin{align}
    \nu_M(x)\;&=\;\cos\theta(x)\nu_K+\sin\theta(x)\nu_\Sigma(x)\,,\label{nuM} \\
    \nu_{\bd(M)}(x)\;&=\;\sin\theta(x)\nu_K-\cos\theta(x)\nu_\Sigma(x)\,,\label{nuMbdM}
  \end{align}
  where $\nu_{\bd(M)}$ denotes the outer normal to $\bd(M)$ with respect to $M$.

\section{Torsion potential with a substrate}\label{sec:tp}
In this section, we will assume hypotheses {\bf(h1)}--{\bf(h4)} from Section~\ref{sec:hyp} and, per the discussion in the Introduction, we will assume further that the substrate is hydrophilic and the container is almost flat. The precise definitions are as follows:

\medskip

 \noindent {\bf Hydrophilic:} Denoting by $\theta_{min}$ and $\theta_{max}$ the extreme values of $\theta$ on $\bd(M)$, we shall assume that there exists $\delta_0>0$ such that
  \begin{equation}
    \begin{split}
      &\mbox{$[\theta_{min},\theta_{max}]\subseteq (\pi/2+\delta_0,\pi)$.}
    \end{split}
  \end{equation}

\medskip

\noindent {\bf Almost Flatness:} There exists $\e\ll1$ such that, up to rotation and translation
\begin{equation}\label{containers conv to half space}
\|A_{\pa K}\|_{C^0(\Sigma)}<\e ,\quad \|1+e_{n+1}\cdot\nu_{K}\|_{C^0(\Sigma)}<\e, \quad\mbox{and}\quad \|x\cdot\nu_{K}\|_{C^0(\Sigma)}<\e |x|.
\end{equation}

\medskip

\begin{remark}\label{rem:c1reg}
    While {\bf Almost Flatness} naturally arises from the proof of Theorem \ref{thm:locmin}, we need to assume we are in the {\bf Hydrophilic} regime due to the fact that the torsion potential fails to be $C^1(\ov{\Omega})$ in the hydrophobic regime, see Lemma~\ref{torsionC1} below. The regularity of the torsion potential is essential to justify Reilly's identity, see Lemma~\ref{lemma Reillys} below.
\end{remark}

\bigskip

We define the \emph{torsion potential of $\Omega$} to be the unique solution $u\in C^2(\Omega\cup\Sigma)\cap C^0(\ov{\Omega})$ to the mixed boundary value problem
\begin{equation}\label{torsion potential}
    \begin{cases}
        \hfill -\Delta u=1&\mbox{ in $\Omega$}\\
        \hfill u=0&\mbox{ on $M$}\\
        \hfill \partial_{\nu_K}u=\gamma&\mbox{ on $\Sigma$},
    \end{cases}
\end{equation}
where
\begin{equation}\label{gamma}
\gamma=-\frac{n}{n+1}\frac{\H^n(\Sigma)}{\int_{\partial\Sigma}\tan\theta(x)\,d\H^{n-1}(x)} \,.
\end{equation}
Existence, uniqueness, and regularity follow from \cite[Theorem 1]{Lieberman}. Next, we show that in the {\bf Hydrophilic} regime one can prove additional regularity of the torsion potential and that in the  {\bf Almost Flatness} regime one can control its $L^\infty$-norm.

\smallskip

\begin{lemma}\label{torsionC1}
Under the hypothesis of Section~\ref{sec:hyp}, assume further that the substrate is {\bf Hydrophilic}. Then, the unique solution $u\in C^2(\Omega\cup\Sigma)\cap C^0(\ov{\Omega})$ to \eqref{torsion potential} additionally satisfies
    \begin{equation}\label{u reg and pos}
        u\in C^1(\overbar{\Omega}) \quad \textnormal{and} \quad u > 0 \;\; \textnormal{in } \Omega \,.
    \end{equation}
Moreover, if $\Sigma$ satisfies {\bf Almost Flatness} with $\e$ small enough then 
    \begin{equation}\label{C^0 est u}
        \|u\|_{C^0(\Omega)} \;\leq\; 4(d^2_\Omega+\gamma^2)\,.
    \end{equation}
\end{lemma}
\begin{proof}
    The $C^1$-regularity of $u$ follows directly from \cite[Thm. 1]{azzam1981regularitatseigenschaften} after using a standard argument to locally straighten the boundary of $\partial\Omega$. 
    
    We now show that $u\geq 0$ by showing that its minimum is achieved on $\pa\Omega\cap K$, where $u$ vanishes. Indeed, the minimum can not be achieved in the interior by super harmonicity, and the minimum of $u$ cannot be achieved on $\Sigma$ by applying Hopf's lemma, since {\bf Hydrophilic} implies
    \begin{equation}
        \partial_\nu u(x) \;=\; \gamma \;\ge \; 0 \,.
    \end{equation}
    
    To show \eqref{C^0 est u}, consider $v\in C^2(\Omega\cup\Sigma)\cap C^1(\ov{\Omega})$ given by
    \begin{equation}
        v(x) := u(x) + \frac{|x|^2}{2(n+1)} + \frac{\gamma+\frac{\e}{n+1} d_\Omega}{1-\e} x_{n+1} \,,
    \end{equation}
    where $\e$ is given by the {\bf Almost Flatness} hypothesis.
    Then, $v$ is the unique solution to
    \begin{equation}\label{torsion potential}
        \begin{cases}
             \hspace{.45cm}-\Delta v=0&\mbox{ in $\Omega$}\\
             \hspace{.56cm} v(x)= \frac{|x|^2}{2(n+1)} + \frac{\gamma+\frac{\e}{n+1} d_\Omega}{1-\e} x_{n+1}&\mbox{ on $M$}\\
            \partial_{\nu_K}v=\gamma+\frac{1}{n+1}x\cdot\nu_K+\frac{\gamma+\frac{\e}{n+1} d_\Omega}{1-\e} e_{n+1}\cdot\nu_K&\mbox{ on $\Sigma$},
        \end{cases}
    \end{equation}
    By the {\bf Almost Flatness} hypothesis, we have that
    $$
        |x\cdot\nu_K|\le \e d_\Omega\qquad\mbox{and}\qquad e_{n+1}\cdot\nu_K\le -(1-\e),
    $$
    which implies that
    $$
        \partial_{\nu_K}v\le 0 \qquad\mbox{ on $\Sigma$.}
    $$
    By harmonicity and Hopf's lemma, we have that $v$ satisfies
    \begin{equation}
    \sup_{\ov{\Omega}} u - 2 (d_\Omega^2+\gamma^2)    \leq \sup_{\ov{\Omega}} v \,\leq\, \sup_M v \,=\,  \sup_{x \in M} \Big(\frac{|x|^2}{2(n+1)}+ \frac{\gamma+\frac{\e}{n+1} d_\Omega}{1-\e}x_{n+1}\Big) \,\leq\, 2 (d_\Omega^2+\gamma^2)\,.
    \end{equation}
    Thus, \eqref{C^0 est u} holds.
\end{proof}

\medskip

The main point of this section is to prove Propositions \ref{prop} and \ref{lem:deficitconvergence2}, which will be fundamental in the proof of our main compactness theorem.
 
\smallskip
 
\begin{proposition}\label{prop}
Under the hypothesis of Section~\ref{sec:hyp}, assume further that the substrate is {\bf Hydrophilic}, that $\Sigma$ satisfies {\bf Almost Flatness} with $\e$ small enough, and that the mean curvature $H_{M}$ is bounded away from zero. Then,
\begin{align}
      &\frac{|\Omega|-\gamma\H^{n}(\Sigma)}{n+1} \left( \int_{M} \frac{n}{H_M}d\H^n - (n+1)\big(|\Omega|-\gamma\H^{n}(\Sigma)\big) \right)\\
      &\qquad+2(n+1)\e \left(\int_M\frac{d\H^n}{H_M}\left(\frac{|\Omega|}{n+1}+4(d_\Omega^2+\gamma^2)(|\Omega|+\gamma \H^n(\Sigma))\right)+\max\frac{1}{H_M}\left(|\Omega|-\gamma\H^n(\Sigma)\right)^2 \right)\\
      &\ge (1-2(n+1)\e)\int_{M}\frac{d\H^n}{H_M} \int_{\Omega}\left(|\nabla^2u|^2-\frac{(\Delta u)^2}{n+1}\right)\;dx\\
      &\qquad+ \int_{M}\frac{d\H^n}{H_M}\int_{M}|\nabla u|^2 (H_M-2(n+1)\e)d\H^n  - \left(1-2(n+1)\e\max\frac{1}{H_M}\right)\left(\int_{M}|\nabla u|d\H^n\right)^2\\
      &\ge 0
      \label{big Reilly with curvature bounds} \,.
\end{align}
\end{proposition} 
\begin{proof}[Proof of Proposition~\ref{prop}]
By Lemma~\ref{lemma Reillys} and Lemma~\ref{boundgrad} we have that
\begin{align}
      &\frac{|\Omega|-\gamma\H^{n}(\Sigma)}{n+1} \left( \int_{M} \frac{n}{H_M}d\H^n - (n+1)\big(|\Omega|-\gamma\H^{n}(\Sigma)\big) \right)\\
      &\ge \int_{M}\frac{d\H^n}{H_M} \int_{\Omega}\left((1-2(n+1)\e)|\nabla^2u|^2-\frac{(\Delta u)^2}{n+1}\right)\;dx\\
      &\qquad+ \int_{M}\frac{d\H^n}{H_M}\int_{M}|\nabla u|^2 (H_M-2(n+1)\e)d\H^n  - \left(\int_{M}|\nabla u|d\H^n\right)^2.
\end{align}
The first inequaliy in \eqref{big Reilly with curvature bounds} now follows by using that $(\Delta u)^2=1$ and the divergence Theorem to obtain
$$
\left(\int_M |\nabla u|\right)^2=(|\Omega|-\gamma\H^n(\Sigma))^2.
$$

To obtain the positivity we notice that by Cauchy-Schwarz
$$
\inf \frac{H_M-C(n)\e}{H_M}\left(\int_{M}|\nabla u|\right)^2\le \left(\int_{M}|\nabla u|\frac{\sqrt{H_M-C(n)\e}}{\sqrt{H_M}}\right)^2\le \int_{M}\frac{1}{H_M}\int_{M}|\nabla u|^2 (H_M-C(n)\e),
$$
with $C(n)=2(n+1)$.
\end{proof}

\bigskip

\begin{proposition}\label{lem:deficitconvergence2}
Under the hypothesis of Section~\ref{sec:hyp}, assume further that the substrate is {\bf Hydrophilic} with constant $\delta_0$, that $\Sigma$ satisfies {\bf Almost Flatness}, and that $\e$, $\|\theta-\theta_0\|_{C^0}$ and $\frac{|H_M-\lambda|_{C_0}}{\lambda}$ are small enough. Then, we have the bound
\begin{align}\label{almost equality}
    &\left|(n+1)(|\Omega|-\gamma\H^n(\Sigma))-\int_M\frac{n}{H}\;d\H^n\right|\\
    &\qquad\le C(\delta_0)\left(\e+\|\theta-\theta_0\|_{C^0(\pa\Sigma)}+\frac{|H_M-\lambda|_{C_0}}{\lambda}\right)\left(d_\Omega+\frac{1}{\lambda}\right)\H^n(M) \,,
\end{align}
    where $\gamma$ is as in \eqref{gamma}.
\end{proposition}
\begin{proof}[Proof of Proposition~\ref{lem:deficitconvergence2}]
We consider the vector field $X=x$, and we apply the divergence theorem to obtain
\begin{equation}\label{proof prop 4.3 (i)}
    (n+1)|\Omega|\;=\;\int_{M}x\cdot\nu\,d\H^n\,+\,\int_\Sigma x\cdot\nu_V\,d\H^n \,.
\end{equation}
By {\bf Almost Flatness}, we have
$$
\left|\int_\Sigma x\cdot\nu_V\;d\H^n\right|\le \e d_\Omega\H^n(\Sigma) \,,
$$
which will become a part of the remainder. For the first term in \eqref{proof prop 4.3 (i)}, the tangential divergence theorem yields the following identity:
\begin{equation}\label{proof prop 4.3 (ii)}
    \int_{M}x\cdot\nu\;d\H^n=\frac{n}{\lambda}\H^n(M)+\frac{1}{\lambda}\int_{\pa\Sigma} x\cdot\nu_{co}^{M}\;d\H^{n-1}+ \int_M\left( 1-\frac{H_M}{
\lambda}\right) (x\cdot\nu)\;d\H^n \,.
\end{equation}
We notice the last term of \eqref{proof prop 4.3 (ii)} can be lumped into the remainder
\begin{equation}
    \left|\int_M\left( 1-\frac{H_M}{\lambda}\right) (x\cdot\nu)\;d\H^n\right|\le d_\Omega\H^n(M)\frac{|H_M-\lambda|_{C_0}}{\lambda} \,.
\end{equation}
For the second term of \eqref{proof prop 4.3 (ii)}, we use the almost constant angle condition to obtain the identity
\begin{align}
\int_{\pa\Sigma} x\cdot\nu_{co}^{M}\;d\H^{n-1}&=-\cos(\theta_0)\int_{\pa\Sigma}  x\cdot \nu_\Sigma \;d\H^{n-1}+\int_{\pa\Sigma} (\cos(\theta_0)-\cos(\theta))x\cdot \nu_\Sigma \;d\H^{n-1} \\ &\hspace{4.95cm}+\int_{\pa\Sigma}\sin(\theta) x\cdot \nu_V \;d\H^{n-1} \,. \label{a1}
\end{align}
The last two terms can be lumped into the remainder by the bounds
$$
\left|\int_{\pa\Sigma} (\cos(\theta_0)-\cos(\theta))x\cdot \nu_\Sigma \;d\H^{n-1}\right|\le  \|\theta-\theta_0\|_{C^0(\pa\Sigma)}d_\Omega \H^{n-1}(\pa\Sigma)
$$
and
$$
\left|\int_{\pa\Sigma}\sin(\theta) x\cdot \nu_V \;d\H^{n-1}\right|\le \e d_\Omega  H^{n-1}(\pa\Sigma).
$$
For the first term in \eqref{a1}, we can use the tangential divergence theorem on $\Sigma$ to obtain
$$
\int_{\pa\Sigma} x\cdot \nu_\Sigma \;d\H^{n-1}=n\H^n(\Sigma)-\int_{\Sigma}H_{\pa K}x\cdot \nu_V\;d\H^n,
$$
and we can bound
$$
\left|\int_{\Sigma}H_{\pa K} x\cdot \nu_V\;d\H^n\right|\le \e^2 d_\Omega\H^n(\Sigma)
$$
Putting all the previous identities together, we obtain that there exists a constant depending on $d_\Omega$, $\H^n(M)$, $\|H_\Sigma\|_{C^0}$ and $\H^{n-1}(\pa\Sigma)$, such that
\begin{align}
&\left|(n+1)|\Omega|+\frac{\cos(\theta_0)n\H^n(\Sigma)}{\lambda}-\frac{n\H^n(M)}{\lambda}\right|\\
&\quad\le \e d_\Omega\left(2\H^n(\Sigma)+\frac{\H^n(\pa\Sigma)}{\lambda}\right)+d_\Omega\H^n(M)\frac{|H_M-\lambda|_{C_0}}{\lambda}+\|\theta-\theta_0\|_{C^0(\pa\Sigma)}d_\Omega \frac{H^{n-1}(\pa\Sigma)}{\lambda}\,.\label{lambdabound}  
\end{align}
Next, using the mean curvature deficit we can bound the difference
$$
\left|\frac{n\H^n(M)}{\lambda}-\int_M\frac{n}{H}\;d\H^n\right|\;\le\;\frac{\|H-\lambda\|_{C^0}}{\lambda} \int_M \frac{n}{H}\;dH^n\;\le\;\frac{\|H-\lambda\|_{C^0}}{\lambda} 2n\frac{\H^n(M)}{\lambda} \,,
$$
where we have assumed that the mean curvature deficit is small enough.

Using Lemma~\ref{boundpasigma}, we have
\begin{align*}
   &\left|\frac{\cos(\theta_0)n\H^n(\Sigma)}{\lambda} - (n+1)\gamma_0\H^n(\Sigma)\right|\\
   &\qquad\le \cot\theta_0\left(\frac{6}{\sin(\theta_0)}(\e+\|\theta-\theta_0\|_{C^0})+\frac{\|H-\lambda\|_{C^0}}{\lambda}\right)\frac{\H^n(M)\H^n(\Sigma)}{\H^{n-1}(\pa\Sigma)},
\end{align*}
where
$$
\gamma_0=\frac{n}{n+1}\frac{\H^n(\Sigma)}{\tan(\theta_0)\H^{n-1}(\pa\Sigma)}.
$$
Using that $\theta$ is bounded away from $\pi/2$ by the hypothesis {\textbf{Hydrophilic}}, we have the bound
$$
|(n+1)\gamma_0\H^n(\Sigma)-(n+1)\gamma|\le C(\delta_0)\frac{\H^n(\Sigma)^2}{\H^{n-1}(\pa\Sigma)}\|\theta-\theta_0\|_{C^0}.
$$
Finally, combining the above inequalities with the bound on $\H^n(\pa\Sigma)$ from Lemma~\ref{boundpasigma}, we obtain the bound in the statement of the Proposition.
\end{proof}

\medskip

\subsection{Supporting lemmas} In this section, we develop the lemmas used in the proofs of Propositions~\ref{prop} and \ref{lem:deficitconvergence2}. We start by exploiting the boundary regularity of the torsion potential $u$ in the hydrophilic regime \eqref{u reg and pos} to prove a Reilly's-type identity.
\begin{lemma}\label{lemma Reillys}
Under the hypothesis of Section~\ref{sec:hyp}, assume further that the substrate is {\bf Hydrophilic}. Then, $u$ satisfies the identities
\begin{equation}\label{Reilly identity}
    \frac{n}{n+1}(|\Omega|-\gamma\H^n(\Sigma))+\int_\Omega \frac{(\Delta u)^2}{n+1}-|D^2u|^2=\int_M (u_\nu)^2 H +\int_\Sigma \gamma^2H_{\pa K} + A_{\pa K}[\nabla^{\Sigma} u, \nabla^{\Sigma} u] \,,
\end{equation}
and
\begin{align}
      \frac{|\Omega|-\gamma\H^{n}(\Sigma)}{n+1}&\left( \int_{M} \frac{n}{H} - (n+1)\big(|\Omega|-\gamma\H^{n}(\Sigma)\big) \right)  \\
      &= \int_{M}\frac{1}{H} \int_{\Omega}\left(|\nabla^2u|^2-\frac{(\Delta u)^2}{n+1}\right) - \left(\int_{M}|\nabla u|\right)^2 \\
      &\hspace{1.1cm}+ \int_{M}\frac{1}{H}\left(\int_{M}|\nabla u|^2H + \int_\Sigma \gamma^2 H_{\pa K} + A_{\pa K}[\nabla^{\Sigma} u,\nabla^{\Sigma} u] \right) \,.\;\;\; \label{big Reilly}
\end{align}

\end{lemma}
\begin{proof}
We re-write the following expression as a divergence:
\begin{equation}
(\Delta u)^2-|D^2u|^2=\nabla\cdot (\nabla u \Delta u-D^2u\cdot \nabla u) \,.
\end{equation}
Hence, by using $-\Delta u=1$, we have
\begin{eqnarray}\label{aux1}
\frac{n}{n+1}|\Omega|+\int_\Omega \frac{(\Delta u)^2}{n+1}-|D^2u|^2
=\int_{\partial\Omega} u_\nu \Delta u- D^2u[\nabla u,\nu] \,.
\end{eqnarray}
On $M$, we can use Reilly's identity. Indeed, because $u=0$ on $M$ and $\nabla u= u_\nu \nu$, we have that pointwise on $M$
\begin{equation}
u_\nu \Delta u- D^2u[\nabla u,\nu]= u_\nu \sum_{i=1}^nD^2u[\tau_i,\tau_i]=u_\nu\Delta_{M}u+ H(u_\nu)^2=H(u_\nu)^2 \,,
\end{equation}
where $\{\tau_i\}_{i=1}^n$ is a basis of the tangent space and $\Delta_{M}$ is the Laplace-Beltrami operator on $M$. Therefore,
\begin{equation}\label{aux2}
\int_M u_\nu \Delta u- D^2u[\nabla u,\nu]=\int_M H(u_\nu)^2 \,.
\end{equation}
Next, given a point $p\in\Sigma$ and a basis $\{\tau_i\}_{i=1}^n$ for $T_p\Sigma$ we find pointwise
\begin{align}
    u_\nu & \Delta u- D^2u[\nabla u,\nu] \\
    &=\; u_\nu \left\{ \sum_{i=1}^n D^2u[\tau_i,\tau_i] + D^2u[\nu,\nu]  \right\} \;-\; \left\{\sum_{i=1}^n \langle \nabla u,\tau_i \rangle \, D^2u[\tau_i,\nu] + \langle \nabla u,\nu \rangle \, D^2u[\nu,\nu] \right\} \\
    \;&=\; u_\nu \, \sum_{i=1}^n D^2u[\tau_i,\tau_i] \;-\; \sum_{i=1}^n \langle \nabla u,\tau_i \rangle \, D^2u[\tau_i,\nu] \\
    \;&=\; u_\nu \, \Delta_\Sigma u \,+\, H\,(u_\nu)^2 \;-\; \sum_{i=1}^n \langle \nabla u,\tau_i \rangle \, D^2u[\tau_i,\nu] \,. \label{div term on Sigma 1}
\end{align}
We now compute the Hessian term in the sum to be
\begin{align}
    D^2u[\tau_i,\nu] 
    \;&=\; \langle D_{\tau_i}\nabla u, \nu \rangle \\
    \;&=\; D_{\tau_i} \langle \nabla u,\nu \rangle \,-\, \langle \nabla u, D_{\tau_i}\nu \rangle \\ 
    \;&=\; 0 \,-\, \langle \nabla^\Sigma u, D_{\tau_i}\nu \rangle \,-\, u_\nu \langle \nu, D_{\tau_i}\nu \rangle \\ 
    \;&=\; \,-\, A_\Sigma[\nabla^\Sigma u, \tau_i] \,,
\end{align}
where we used the facts that $u_\nu$ is constant on $\Sigma$ (and hence $D_{\tau_i}u_\nu \equiv 0$) and that \begin{equation}
    \langle \nu, D_{\tau_i} \nu \rangle = D_{\tau_i} \langle \nu,\nu \rangle - \langle D_{\tau_i} \nu , \nu \rangle = - \langle D_{\tau_i} \nu , \nu \rangle \,,
\end{equation}
which directly implies that $\langle \nu, D_{\tau_i} \nu \rangle \equiv 0$. Thus, we can rewrite \eqref{div term on Sigma 1} to be 
\begin{align}
    u_\nu \Delta u- D^2u[\nabla u,\nu]\;&=\; u_\nu \, \Delta_\Sigma u \,+\, H(u_\nu)^2 \;-\; \sum_{i=1}^n \langle \nabla u,\tau_i \rangle \, D^2u[\tau_i,\nu] \\
    \;&=\; u_\nu \, \Delta_\Sigma u \,+\, H(u_\nu)^2 \,+\, \sum_{i=1}^n \langle \nabla u,\tau_i \rangle \, A_\Sigma[\nabla^\Sigma u, \tau_i] \\
    \;&=\;  u_\nu \, \Delta_\Sigma u \,+\, H(u_\nu)^2 \,+\, A_\Sigma[\nabla^\Sigma u, \nabla^\Sigma u]
\end{align}
Hence, by the divergence theorem on $\Sigma$, we have
\begin{align}
    \int_\Sigma u_\nu \Delta u- D^2u[\nabla u,\nu] 
    \;&=\; \gamma \int_\Sigma \Delta_\Sigma u + \gamma^2 H_{\Sigma} + A_\Sigma[\nabla^\Sigma u,\nabla^\Sigma u] \\ 
    \;&=\; \gamma \int_{\pa\Sigma}\nabla u\cdot\nu_\Sigma \,+\, \int_\Sigma \gamma^2 H_\Sigma + A_\Sigma[\nabla^\Sigma u,\nabla^\Sigma u] \,. \label{temp id 3}
\end{align}
We now use the definition of the angle to write the co-normal of $M$ on $\partial\Sigma$ as
\begin{equation}
\nu_{co}=-\cos(\theta)\nu_\Sigma+\sin(\theta)\nu_K \,.
\end{equation}
As noted above, we have that $\nabla u = u_\nu \nu$ on $M$, and hence on $\partial M = \partial\Sigma$ by the $C^1$ regularity on $\ov{\Omega}$. So, we find that pointwise on $\partial\Sigma$
\begin{equation}
0=\nabla u\cdot \nu_{co}=-\cos(\theta)\nabla u\cdot \nu_\Sigma+\gamma \sin(\theta) \,.
\end{equation}
We thus arrive at the pointwise identity 
\begin{equation}
\nabla u\cdot \nu_\Sigma=\gamma\frac{\sin(\theta)}{\cos(\theta)} \quad \textnormal{on } \partial\Sigma \,,
\end{equation}
which we substitute into \eqref{temp id 3} to find
\begin{align}
    \int_\Sigma u_\nu \Delta u- D^2u[\nabla u,\nu]
    \;&=\; \gamma^2\int_{\partial\Sigma}\frac{\sin(\theta)}{\cos(\theta)} \,+\, \int_\Sigma \gamma^2 H_\Sigma+ A[\nabla^\Sigma u,\nabla^\Sigma u] \\
    \;&=\; \gamma \frac{n}{n+1}\H^n(\Sigma) \,+\, \int_\Sigma \gamma^2 H_\Sigma + A[\nabla^\Sigma u,\nabla^\Sigma u] \,, \label{aux3}
\end{align}
where the last equality follows from the definition of $\gamma$. The identity \eqref{Reilly identity} then follows by putting together \eqref{aux1}, \eqref{aux2} and \eqref{aux3}.

By the divergence theorem one has
    \begin{equation}
      |\Omega|-\gamma\H^n(\Sigma) = \int_M|\nabla u| \,,
    \end{equation}
    so by using this identity and \eqref{Reilly identity} we can directly expand the right-hand-side of \eqref{big Reilly} to find
    \begin{align}
      &\int_{M} \frac{1}{H} \int_{\Omega}\left(|\nabla^2u|-\frac{(\Delta u)^2}{n+1}\right) + \int_{M}\frac{1}{H}\left(\int_{M}|\nabla u|^2H + \int_\Sigma \gamma^2 H + A[\nabla^{\Sigma} u,\nabla^{\Sigma} u] \right) - \left(\int_{M}|\nabla u|\right)^2 \\
      &= \int_{M} \frac{1}{H} \left(\int_{\Omega}|\nabla^2u|-\frac{(\Delta u)^2}{n+1} + \int_M |\nabla u|^2H + \int_\Sigma \gamma^2 H + A[\nabla^{\Sigma} u,\nabla^{\Sigma} u] \right) - \left(\int_{M}|\nabla u|\right)^2 \\
      &= \left(\int_{M} \frac{1}{H} \right) \frac{n}{n+1} \big(|\Omega|-\gamma\H^{n-1}(\Sigma)\big)   \;-\; \big(|\Omega|-\gamma\H^{n}(\Sigma)\big)^2 \\
      &= \frac{|\Omega|-\gamma\H^{n}(\Sigma)}{n+1} \left( \int_{M} \frac{n}{H} - (n+1)\big(|\Omega|-\gamma\H^{n}(\Sigma)\big) \right) \,,
    \end{align}
    which is \eqref{big Reilly}.
\end{proof}

\medskip

Next, we utilize the {\bf Almost Flatness} hypothesis to control the extra the terms in the wetted region.
\begin{lemma}\label{boundgrad}
Under the hypothesis of Section~\ref{sec:hyp}, assume further that the substrate is {\bf Hydrophilic}, $\Sigma$ satisfies {\bf Almost Flatness} with parameter $\e<1/2$, then 
\begin{align}
      &\left|\int_\Sigma \big( \gamma^2 H_\Sigma + A_\Sigma[\nabla^\Sigma u,\nabla^\Sigma u] \big) \,d\H^n \right|\\ 
      \;&\qquad\leq\; 2(n+1)\e \left(\int_\Omega\big(|D^2u|^2+|\nabla u|^2\big)dx+\int_M |\nabla u|^2\;d\H^n \right)\\
      \;&\qquad\leq\; 2(n+1)\e \left(4(d_\Omega^2+\gamma^2)(|\Omega|+\gamma\H^n(\Sigma))+\int_\Omega |D^2u|^2\;dx+\int_M |\nabla u|^2\;d\H^n \right)\,, \;\;\;\;\;\;\;\; \label{est on A[Du,Du]}
\end{align}
     where $u\in C^2(\Omega\cup\Sigma)\cap C^1(\ov{\Omega})$ is the unique solution to \eqref{torsion potential}.
\end{lemma}
\begin{proof} We consider the vector field $-|\nabla u|^2 e_{n+1}$, and apply the divergence theorem to obtain
$$
\int_{\Sigma} |\nabla u|^2 (-e_{n+1}\cdot\nu_\Sigma)\,d\H^{n}=\int_{M} |\nabla u|^2 e_{n+1}\cdot\nu_M\;d\H^n+2\int_\Omega \nabla u\cdot \nabla \pa_{n+1}u\;dx
$$
From the {\bf Almost Flatness} hypothesis, we obtain
\begin{equation}\label{e1}
(1-\e)\int_{\Sigma} |\nabla u|^2 \,d\H^{n}\le \int_{M} |\nabla u|^2\;d\H^n+\int_\Omega |D^2u|^2+|\nabla u|^2\;dx \,.
\end{equation}
Using the previous inequality with $\e<1/2$, we can directly bound
    \begin{align}
        \left|\int_\Sigma \big( \gamma^2 H_\Sigma + A_\Sigma[\nabla u,\nabla u] \big) \,d\H^n \right| 
        \;&\leq\; (n+1)\,\|A_\Sigma\|_{C^0(\Sigma)}\, \int_\Sigma |\nabla u|^2\,d\H^n \\
        \;&\leq\; 2  (n+1)\,\|A_\Sigma\|_{C^0(\Sigma)}\, \left(\int_\Omega\big(|D^2u|^2+|\nabla u|^2\big)dx + \int_M |\nabla u|^2\,d\H^n \right) \,.
    \end{align}
    The first inequality in \eqref{est on A[Du,Du]} then follows from the {\bf Almost Flatness} hypothesis $\|A_\Sigma\|_{C^0(\Sigma)}<\e$. For the second inequality, we integrate by parts and use \eqref{torsion potential} to obtain
    $$
        \int_\Omega |\nabla u|^2\;dx= \int_\Omega u\;dx+\gamma \int_\Sigma u\;d\H^n\le \|u\|_{\infty}(|\Omega|+\gamma \H^n(\Sigma)) \,,
    $$
    and the desired inequality now follows from applying the $L^\infty$-estimate \eqref{C^0 est u}.
\end{proof}

Finally, we use the {\bf Hydrophilic} and {\bf Almost Flatness} hypotheses to control the perimeter of the wetted region.

\begin{lemma}\label{boundpasigma}
Under the hypothesis of Section~\ref{sec:hyp}, assume further that the substrate is {\bf Hydrophilic}, $\Sigma$ satisfies {\bf Almost Flatness}, and that $\e$ and $\|\theta-\theta_0\|_{C^0}$ are small enough. Then, we have the bounds
$$
        \H^{n-1}(\pa\Sigma)\le \frac{2}{\sin(\theta_0)}\left(\lambda\H^n(\Sigma)+\|H-\lambda\|_{C^0}\H^n(M)\right)
$$
and
$$
    \left|\frac{\cos\theta_0}{\lambda}-\frac{\H^n(\Sigma)}{\tan\theta_0 \H^{n-1}(\pa \Sigma)}\right|\le \cot\theta_0\left(\frac{6}{\sin(\theta_0)}(\e+\|\theta-\theta_0\|_{C^0})+\frac{\|H-\lambda\|_{C^0}}{\lambda}\right)\frac{\H^n(M)}{\H^{n-1}(\pa\Sigma)} \,.
$$
\end{lemma}
  \begin{proof}
    Using the divergence theorem on the vector field $X=e_{n+1}$, we have that
      \begin{equation}\label{temp formula}
          0=\int_M \Div X\;d\H^n=\int_{M} H X\cdot \nu_\Omega\;d\H^n+\int_{\pa \Sigma} \big( \sin\theta X\cdot \nu_K-\cos\theta X\cdot \nu_\Sigma\big)\;d\H^{n-1} \,.
      \end{equation}
      For the first term of \eqref{temp formula}, we use the almost constancy of the mean curvature to obtain
      $$
      \int_{M} H X\cdot \nu_\Omega\;d\H^n\le \|H-\lambda\|_C^0\H^n(M)+\lambda \int_{\Sigma}X\cdot \nu_K\;d\H^n \,,
      $$
      and then, by using {\bf Almost Flatness}, we further find that
      $$
        \left|\lambda\H^n(\Sigma)-\int_{M} H X\cdot \nu_\Omega\;d\H^n\right|\le \e \lambda\H^n(\Sigma) +\|H-\lambda\|_{C^0}\H^n(M) \,.
      $$
      For the second term of \eqref{temp formula}, we also use {\bf Almost Flatness} to obtain
      \begin{align*}
        &\left|\sin\theta_0 \H^{n-1}(\pa \Sigma)+\int_{\pa \Sigma}\big( \sin\theta X\cdot \nu_K-\cos\theta X\cdot \nu_\Sigma \big)\;d\H^{n-1}\right| \\
        &\hspace{5cm}\le \H^{n-1}(\pa \Sigma) \|\theta-\theta_0\|_{C^0}+ \e \H^{n-1}(\pa\Sigma) \,.
      \end{align*}
      Hence, we find
      \begin{align}
       \left|\sin\theta_0 \H^{n-1}(\pa \Sigma)+\lambda\H^n(\Sigma)\right| \;&\le\; \e\lambda \H^n(\Sigma) +\|H-\lambda\|_{C^0}\H^n(M) \\  \;&\hspace{2.13cm}+(\e+\|\theta-\theta_0\|_{C^0}) \H^{n-1}(\pa\Sigma) \,, \label{1b}
      \end{align}
      which, for $\e$ and $\|\theta-\theta_0\|_{C^0}$ small enough, implies
      $$
        \H^{n-1}(\pa\Sigma)\le \frac{2}{\sin(\theta_0)}\Big(\lambda\H^n(\Sigma)+\|H-\lambda\|_{C^0}\H^n(M)\Big) \,.
      $$
      Substituting this bound into \eqref{1b}, we get that
      $$
            \left|\sin\theta_0 \H^{n-1}(\pa \Sigma)+\lambda\H^n(\Sigma)\right|\le 4 \left(\frac{2}{\sin(\theta_0)}\lambda(\e+\|\theta-\theta_0\|_{C^0})+\|H-\lambda\|_{C^0}\right)\H^n(M) \,.
      $$
      The desired estimate now follows from dividing by $\sin\theta_0 \H^{n-1}(\pa \Sigma)\lambda$ and multiplying by $\cos\theta_0$.
  \end{proof}

\section{Compactness}\label{sec:compactness}
\subsection{GMT preliminaries}\label{sec:GMT}
In the proof of Theorem \ref{mainthm}, we will make use of the theory of varifolds. Additionally, we will weaken the technical hypothesis of the convergence in perimeter \eqref{hyp: conv in per} by appealing to the deep result of White \cite[Thm. 1.2]{whiteCurrents}, which relates the limits of ``compatible'' sequences of currents and varifolds. Here, we follow the notation of \cite[Chapters 6 and 8]{simon1983lectures}.

\medskip

Let $K$ be a bounded open set with $C^2$-boundary, and let $\Om\subseteq K$ be a bounded connected open set such that $M=\ov{K\cap\pa\Om}$ is a smooth hypersurface with boundary satisfying the hypotheses {\bf(h1)}--{\bf(h4)}. Denote by $M$ and $\Sigma$ the sets $M=\partial\Omega\cap K$ and $\Sigma = \partial\Omega\cap\pa K$. Then, we will let $V$, and $W$ denote the multiplicity one varifolds associated to $M$ and $\partial\Omega$, respectively; that is, for $\psi\in C^0_c\big(\R^{n+1}\times G(n+1,n)\big)$ we define
\begin{align}
  V[\psi] \,&:=\, \int_{M} \psi(p,T_pM)d\H^n \label{def V grass} \,, \\
  W[\psi] \,&:=\, \int_{M} \psi(p,T_pM)d\H^n +\int_{\Sigma} \psi(p,T_p\Sigma)d\H^n \label{varifold assoc bd Omega} \,.
\end{align}
In terms of notation, we sometimes take test functions $\psi\in C^0_c\big(\R^{n+1}\times G(n+1,n)\big)$ that are constant on the second variable $\psi(p,\cdot) \,=\, \varphi(p)\,,$ where $\varphi\in C^0_c(\R^{n+1})$. Through this consideration, for $\varphi\in C^0_c(\R^{n+1})$ we denote
\begin{align}
  V[\varphi] \,&:=\, \int_{M} \varphi\,d\H^n \label{def V no grass} \,.
\end{align}
We will reserve the symbols $\varphi$ for test functions $\varphi\in C^0_c(\R^{n+1})$ which only depend on the space variable and $\psi$ for general test functions $\psi\in C^0_c\big(\R^{n+1}\times G(n+1,n)\big)$.

Next, given $X\in C^\infty_c(\R^{n+1};\R^{n+1})$, the first variation of $V$ and $W$ are given by 
\begin{align}
  \delta V[X] \;&:=\; \int_{M} \Div_\tau X d\H^n=\int_{M} H_M X\cdot\nu_{M} d\H^n + \int_{\bd(M)} X\cdot\nu_{\bd(M)} d\H^{n-1}\\
  \delta W[X] \;&:=\; \int_{M} \Div_\tau X d\H^n+\int_{\Sigma} \Div_\tau X d\H^n=\int_{M} H_M X\cdot\nu_{M} d\H^n + \int_{\bd(M)} X\cdot(\nu_{\bd(M)}+\nu_{\Sigma}) d\H^{n-1}\,, \label{first var gen V}
\end{align}
where $\Div_\tau$ denotes the tangential divergence operator on $M$. Note that the representation on the right-hand-side follows from the regularity we have assumed, see Section~\ref{sec:hyp}.

\medskip

We now state a simplified version of \cite[Thm. 1.2]{whiteCurrents} sufficient for our setting:

\begin{theorem}\label{white's theorem}
  Let $\{\Omega_l\}_{l\in\N}$ be a sequence of compact sets with uniformly bounded finite perimeter and $\{W_l\}_{l\in\N}$ be the natural multiplicity one varifolds associated to the boundaries $\{\pa\Omega_l\}_{l\in\N}$ as in \eqref{varifold assoc bd Omega}.
  
  If the varifolds $\{W_l\}_{l\in\N}$ have uniformly bounded first variation, then, up to a subsequence, one has that $\Omega_l \to \Omega_\infty$ in $L^1$ and 
  \begin{equation}\label{white's thm representation}
      W_l \;\to\; W_\infty \,=\, \Xi_1\,d\H^{n}|_{\pa \Omega_\infty} \,+\, \Xi_0\,d\H^{n}|_{S} \,,
  \end{equation}
  where $S$ is a rectifiable set disjoint from $\partial\Omega_\infty$ and $\Xi_1$ and $\Xi_0$ are integer valued Borel functions satisfying $\Xi_1=1\pmod{2}$ and $\Xi_0=0\pmod{2}$.
\end{theorem}

\subsection{Proof of Theorem~\ref{mainthm}}
We now state and prove Theorem \ref{mainthm} in its precise form:

\begin{manualtheorem}{1.6}
Let $\{K_l\}_{l\in\N}$ be a sequence of connected open sets with connected $C^2$-boundaries that converge to the half space in the following sense:
\begin{equation}\label{containers conv to half space}
\|A_l\|_{C^0_{loc}(\pa K_l)}\to 0,\qquad \|e_{n+1}\cdot\nu_{K_l}\|_{C^0_{loc}(\pa K_l)}\to 0\qquad\mbox{and}\qquad \|x\cdot\nu_{K_l}\|_{C^0_{loc}(\pa K_l)}\to 0\,,
\end{equation}
where $A_l$ is the second fundamental form of $\pa K_l$.

If $\{\Omega_l\}_{l\in\N}$ is a sequence of connected open sets which satisfy uniform bounds in diameter and perimeter,
     \begin{equation}\label{hyp: per bnd}
     \sup_{l\in\N}\diam(\Omega_l)<\infty \,, \qquad\sup_{l\in\N} P(\Omega_l;\{x_{n+1}>0\})<\infty \,,
     \end{equation}
then, up to subsequence, there exists a bounded set of finite perimeter $\Omega_\infty\subseteq\{x_{n+1}>0\}$ such that $|\Omega_l\Delta\Omega_\infty|\to0$ and for which the following holds: if
the boundaries $\{M_l\}_{l\in\N}$ satisfy hypotheses {\bf(h1)}-{\bf(h4)} from Section~\ref{sec:hyp}, if there exists a positive constant $\lambda_0>0$ such that
    \begin{equation}\label{hyp: conv mc}
        \left\|H_l-\lambda_0\right\|_{C^0(M_l)} \to 0 \,,
    \end{equation}
if there exists a fixed angle $\theta_0\in(\pi/2,\pi)$ such that
\begin{equation}\label{hyp: conv angle}
    \|\theta_l-\theta_0\|_{C^0(\bd(M_l))}\to 0\,,
  \end{equation}
and if the set $S$ in \eqref{white's thm representation} satisfies
 \begin{equation}\label{hyp: conv in per2}
      \H^n(S \cap \{x_{n+1}>0\}) = 0 \,,
 \end{equation}
then $\Omega_\infty$ is equal to a finite union of disjoint tangent balls of radius $n/\lambda_0$ intersected with $\{x_{n+1}>0\}$, the balls in $\Omega_\infty$ which intersect the hyperplane $\{x_{n+1}=0\}$ do so with contact angle $\theta_0$, and
\begin{equation}
    P(\Omega_l;K_l)\to P(\Omega_\infty;\{x_{n+1}>0\}) \,.
\end{equation}
Moreover, the wetted regions converge in the following sense:
    \begin{equation}\label{eq:a2}
    \lim_{l\to\infty}  |\H^n(\Sigma_l)-\H^n(\Sigma_\infty)|+|\H^{n-1}(\pa\Sigma_l)-\H^{n-1}(\pa\Sigma_\infty )| \;=\; 0 \,.
    \end{equation}
\end{manualtheorem}

\begin{remark}
    The technical hypothesis \eqref{hyp: conv in per2} is only used in the first step of the proof to establish \eqref{supp V_infty = pa Omega_infty} and may be able to be removed using additional arguments.
\end{remark}


\begin{proof}[Proof of Theorem~\ref{mainthm}]
Without loss of generality, we may simplify our notation by considering the specific case of $\lambda=n$.

\textit{Notation.} We will let $C$ denote a generic constant independent of $l$ which may depend upon $n$, $R$, $\theta_0$, $\sup_{l\in\N}|\gamma_l|$, and $\sup_{l\in\N}P(\Omega_l)$ and which we will allow to change line-by-line. For all $l\in\N$, we will let $V_l$ denote the multiplicity-one varifold induced by $M_l$.  Next, for each $l\in\N$, we will let $u_l:\Omega_l\to\R$ denote the torsion potential of $\Omega_l$, that is to say the unique solution to
\begin{equation}\label{tp mid proof}
  \left\{\begin{split}
    -\Delta u_l=1\;\;&\qquad\mbox{in $\Om_l$}
    \\
    u_l=0\;\;&\qquad\mbox{on $M_l$}
    \\
    \pa_{\nu_{K_l}}u_l=\gamma_l&\qquad\mbox{on $\Sigma_l$}
  \end{split}\right .
\end{equation}
with $\gamma_l$ as in \eqref{gamma}. By an abuse notation, we will also denote $u_l:K_l\to\R$ to be the zero extension of $u_l$ to $K_l$.

\medskip

\textit{Step one.} By hypothesis there exists $\Omega_\infty \subseteq \{x_{n+1}\ge 0\}$ a set of finite perimeter such that
\begin{align}
    &|\Omega_l \Delta \Omega_\infty| \,\to\, 0 \,. \label{conv in volume}
\end{align}
We now claim that there exists an integral varifold $V_\infty$ such that (up to a subsequence)
\begin{align}
    &V_l \weakstar V_\infty \,, \quad\;\; {\rm and } \quad \delta V_l \weakstar \delta V_\infty \label{varifold conv}
\end{align}
and, moreover, there exists a set of finite perimeter $\Sigma_\infty\subset\{x_{n+1}=0\}$, such that, up to passing to a subsequence, one has the convergence of the wetted region
\begin{align}
    &d\H^n|\Sigma_l\weakstar d\H^n|\Sigma_\infty\label{conv in wetted}
\end{align}
as $l\to\infty$.

We claim that the sequence of varifolds $\{V_l\}_{l\in\N}$ have uniformly bounded mass and first variation. That they have uniformly bounded mass follows directly from \eqref{hyp: per bnd}. For the first variation, we consider $X\in C^\infty_c(\R^{n+1};\R^{n+1})$ and compute
\begin{equation}\label{first variation}
  \delta V_l[X] = \int_{M_l}\Div_\tau X d\H^n = \int_{M_l} H_l X\cdot\nu_{M_l} d\H^n + \int_{\bd(M_l)} X\cdot\nu_{\bd(M_l)} d\H^n \,,
\end{equation}
where $\nu_{M_l}$, $\nu_{\bd(M_l)}$ denote the outer unit normal of $M_l$, and the outer unit conormal of $\bd(M_l)$. Therefore, by \eqref{hyp: per bnd}, Lemma~\ref{boundpasigma}, and \eqref{containers conv to half space}, we have the uniform bound 
\begin{align}
  &\sup_{l\in\N} |\delta V_l|[X] \;\leq\; C \|X\|_{C^0(\R^{n+1})} 
\end{align}
so that by the compactness theorem for integral varifolds \cite[42.7]{simon1983lectures} there exists an integral varifold $V_\infty$ such that, up to a subsequence, \eqref{varifold conv} holds. The convergence in \eqref{conv in wetted} follows by similar arguments.

We conclude this step by noting that by \eqref{hyp: conv in per2} it directly follows that
\begin{equation}\label{supp V_infty = pa Omega_infty}
    \supp V_\infty \cap \{x_{n+1}>0\} \;=\; \overline{\partial^*\Omega_\infty} \cap \{x_{n+1}>0\} \;=\;\partial\Omega_\infty \cap \{x_{n+1}>0\} \,,
\end{equation}
where the last equality is obtained by choosing the representative $\Omega_\infty$ such that its topological boundary coincides with the closure of the reduced boundary $\overline{\partial^*\Omega_\infty}=\partial\Omega_\infty$, see \cite[Proposition 12.19]{maggi2012sets}.

\bigskip

\textit{Step two.}
We now claim that, up to passing to a subsequence, for every $\eta>0$ the following limit holds:
\begin{equation}\label{hd conv varifolds}
  \lim_{l\to\infty} \hd\left(M_l\cap \{x_{n+1}>\eta\}, \, \partial\Omega_\infty \cap \{x_{n+1}>\eta\right) = 0 \,.
\end{equation}
This will follow from a uniform density estimate stemming from the bounded mean curvature hypotheses. By \eqref{containers conv to half space}, \eqref{hyp: conv mc}, and the monotonicity formula for integral varifolds \cite[17.7]{simon1983lectures}, there exists a constant $C>0$ independent of $l$ and $\eta$ such that for all $\eta>0$ and $l$ large enough such that $\Sigma_l\subseteq\{x_{n+1}<\eta/2\}$ the following holds: for any $p\in M_l\cap \{x_{n+1}>\eta\}$ and $\rho\in(0,\eta/2)$ one has
\begin{equation}\label{density est}
  \H^n\big( M_l\cap B_\rho(p) \big) \geq \omega_n\rho^n \,.
\end{equation}
We will now prove that 
\begin{equation}\label{hd inc 1}
  \lim_{l\to\infty}\sup_{p\in M_l\cap \{x_{n+1}>\eta\}} \dist(p,\partial\Omega_\infty)=0.
\end{equation}
Suppose for sake of contradiction that \eqref{hd inc 1} does not hold. Then, there exists $\eta\,,\e>0$ and a sequence of points $\{p_l\}$ such that  $p_l \in M_l\cap \{x_{n+1}>\eta\}$ and $\dist(p_l,\partial\Omega_\infty)\geq\e$. By the uniform diameter of the sets $\{\Omega_l\}_{l\in\N}$, the sequence $\{p_l\}_{l\in\N}$ has a convergent subsequence, which we will not relabel. Let $p_\infty$ denote this limit, and note that $\dist(p_\infty,\partial\Omega_\infty)\ge\e$. Let $\varphi\in C^\infty_c(B_{\e}(p_\infty))$ be such that $0\leq\varphi\leq1$ and $\varphi\equiv1$ on $B_{\e/2}(p_\infty)$. By the convergence of the sequence $p_l\to p_\infty$, for $l$ large enough we have $B_{\e/4}(p_l)\subseteq B_{\e/2}(p_\infty)$. Up to taking $l$ larger, so that we ensure $\Sigma_l\subseteq\{x_{n+1}<\eta/2\}$ by \eqref{containers conv to half space}, for $l$ large enough the density estimate \eqref{density est} yields
\begin{equation}
  C\;\frac{\min\{\e^n,\eta^n\}}{4^n} \;\leq\; \int_{M_l} \varphi\,d\H^n  \;=\; V_l[\varphi]\,.
\end{equation}
However, this contradicts the fact that $V_l[\varphi]\to V_\infty[\varphi]=0$ by \eqref{varifold conv}, where we know that $V_\infty[\varphi]=0$ by $\dist(p_\infty,\partial\Omega_\infty)>\e$ and \eqref{supp V_infty = pa Omega_infty}. Hence, \eqref{hd inc 1} must hold.

To prove the converse limit,
\begin{equation}\label{hd inc 2}
  \lim_{l\to\infty}\sup_{p\in \pa\Omega_\infty\cap \{x_{n+1}>\eta\}} \dist(p,M_l)=0 \,,
\end{equation}
we again argue by contradiction. Suppose that there exists $\eta,\e>0$ and $p\in\partial\Omega_\infty\cap \{x_{n+1}>\eta\}$ such that for all $l$ large enough one has $\dist(p,M_l)>\e$. Let $\varphi\in C^\infty_c(B_{\e/2}(p))$ be such that $0\leq\varphi\leq1$ and $\varphi=1$ on $B_{\e/4}(p)$. Then, by \eqref{varifold conv}, \eqref{supp V_infty = pa Omega_infty}, and the fact that $\dist(p,M_l\cap \{x_{n+1}>\eta\})>\e$ for all $l$, one has
\begin{equation}
  0 < V_\infty[\varphi] = \lim_{l\to\infty} V_l[\varphi] = 0 \,,
\end{equation}
a contradiction. Therefore, we conclude that \eqref{hd conv varifolds} holds as claimed. 

\bigskip

\textit{Step three.} For each $l\in\N$, let $u_l:\Omega_l\to\R$ be the unique solution to \eqref{tp mid proof}. Then, we claim that
\begin{equation}\label{conv of grad and hess u}
  \lim_{l\to\infty} \int_{\Omega_l} \left| \nabla^2 u_l - \frac{\Id}{n+1} \right|^2dx \;=\; 0
\end{equation}
and
\begin{equation}\label{unif estimates on u}
    \sup_{l\in\N}  \|u_l\|_{H^1(K_l)} \;<\; \infty \,.
\end{equation}
Combining Propositions ~\ref{prop}, \ref{lem:deficitconvergence2} and Lemma~\ref{boundpasigma}, we can use the hypothesis in the Theorem to conclude that
\begin{align*}
    0\;=\;
    &\lim_{l\to\infty} \bigg[(1-2(n+1)\e)\int_{M_l}\frac{d\H^n}{H_{M_l}} \int_{\Omega_l}\left(|\nabla^2u_l|^2-\frac{(\Delta u_l)^2}{n+1}\right)\;dx\\
    &\qquad\quad+ \int_{M_l}\frac{d\H^n}{H_{M_l}}\int_{M_l}|\nabla u_l|^2 (H_{M_l}-2(n+1)\e)d\H^n  \\
    &\qquad\quad- \left(1-2(n+1)\e\max\frac{1}{H_{M_l}}\right)\left(\int_{M_l}|\nabla u_l|d\H^n\right)^2\bigg] \,.   
\end{align*}
Using that the two last terms in the previous limit are always positive, since the mean curvature is bounded above and the perimeter bounded below, we may conclude that
$$
\lim_{l\to\infty} \int_{\Omega_l}\left(|\nabla^2u_l|^2-\frac{(\Delta u_l)^2}{n+1}\right)\;dx=0.
$$
Considering the matrices $I$ and $\nabla^2u_l$ as vectors in $v\,,w\in\R^{(n+1)^2}$, we notice that the above integrand can be written as
$$
|\nabla^2u_l|^2-\frac{(\Delta u_l)^2}{n+1}=|w|^2-\frac{v\cdot w}{|v|^2}
$$
and hence is exactly quantifying the deficit in the inequality of Cauchy-Schwarz. Following the arguments in \cite[Proposition 3.7]{delgadinomaggimihailaneumayer}, we conclude that
$$
\lim_{l\to\infty} \frac{1}{2}\int_{\Omega_l}\left|\nabla^2u_l -\frac{\Id}{n+1}\right|^2\;dx=0.
$$

The estimate \eqref{unif estimates on u} follows directly by integration by parts, as the equation and the boundary conditions for $u_l$ yield
\begin{align*}
    \int_{K_l}|\nabla \tilde{u}_l|^2\;dx
    \;&=\; \int_{\Omega_l}|\nabla u_l|^2\;dx \\
    \;&=\; \int_{\Omega_l}u_l\;dx+\gamma_h\int_{\Sigma_l}u_l\;d\H^n \\
    \;&\le\; \|u_l\|_{L^\infty} (|\Omega_l|+\gamma_l \H^n(\Sigma_l)) \,,  
\end{align*}
where the uniform $L^\infty$-bound is given by \eqref{C^0 est u}.

\bigskip

\textit{Step four.}  By \eqref{unif estimates on u} and the asymptotic flatness of the domain, there exists a positive function $u_\infty \in H^1(\{x_{n+1}>0\})$ such that 
\begin{equation}\label{conv u_h to u_infty}
    u_l \to u_\infty  \quad \textnormal{strongly in } L^2_{loc}(\{x_{n+1}>0\})  \quad \textnormal{and} \quad
    u_l \rightharpoonup u_\infty \quad \textnormal{weakly in } H^1_{loc}(\{x_{n+1}>0\}) \,.
\end{equation}
We now claim that
\begin{equation}\label{hessID}
    D^2 u_\infty= \frac{\Id}{n+1}\qquad\mbox{in $\Omega_\infty$}
\end{equation}
and, moreover, there exists a finite number of disjoint balls $\{B_{i}(y_i)\}_{i\in I}$ with radii satisfying $r_i=1\pmod2$  such that
\begin{equation}\label{Omega is balls}
    \Omega_\infty\cap\{x_{n+1}>0\} \;=\; \bigcup_{i\in I} B_i\cap \{x_{n+1}>0\} \,.
\end{equation}

\medskip

We begin by showing \eqref{hessID} holds in distribution. Let $z\in \Omega_\infty\cap \{x_{n+1}>0\}$, and let $\e>0$ be such that $d(z,\pa\Omega_\infty)>\e>0$. By the Hausdorff convergence \eqref{hd conv varifolds}, we can take $l$ large enough such \begin{equation}
    d\big(\pa\Omega_l\cap \{x_{n+1}>\eta\},\pa \Omega_\infty\cap \{x_{n+1}>\eta\}\big)\le \e/2
\end{equation}
for $\eta=z_{n+1}/2$, which implies that $B_{\e/2}(z)\subseteq \Omega_l$ for $l$ large enough. By \eqref{conv of grad and hess u} and the previous inclusion, one can directly see that for any $\phi\in C^\infty_c(B_{\e/2}(z))$
\begin{equation}
    \frac{\Id}{n+1}\int \phi\,dx \;=\; \lim_{l\to \infty}\int D^2u_l\phi\,dx \;=\; \lim_{l\to \infty}\int u_l D^2\phi\,dx\;=\;\int u_\infty D^2\phi\;dx \,,
\end{equation}
which shows \eqref{hessID} in distribution. This readily implies that $u_\infty\in C^\infty(\Omega_\infty)$ and that its nodal set has zero Lebesgue measure, i.e.  $|\{u_\infty=0\}\cap\Omega_\infty|=0$. Thus, in each connected component of $\Omega_\infty=\bigcup \Omega_\infty^i$, there exists $y_i \in \R^{n+1}$ and $r_i\in (0,\infty)$ such that
\begin{equation}\label{explicitvalue}
    u_\infty(x)=\frac{|x-y_i|^2-r_i^2}{2(n+1)} \qquad \textnormal{in }\;  \Omega_\infty^i \,.
\end{equation}
Next, using the weak $H^1$ convergence and integration by parts, we have that for any test function $\vphi\in C^\infty_c(\{x_{n+1}>0\})$
$$
\int_{\{x_{n+1}>0\}}\nabla u_\infty \phi\;dx\;=\;\lim_{l\to\infty}\int_{\Omega_l}\nabla u_l \phi\;dx\;=\;-\lim_{l\to\infty}\int_{\Omega_l}u_l \nabla \phi\;dx\;=\;-\int_{\Omega_\infty} u_\infty \nabla \phi\;dx \,,
$$
which by the uniform continuity of $u_\infty$ on $\Omega_\infty$ implies that the trace vanishes
$$
u_\infty=0 \qquad\mbox{on $\pa\Omega_\infty=0$}\,.
$$
We therefore may directly conclude that each component 
\begin{equation}\label{int characterization of Om}
    \Omega_\infty^i=B_{r_i}(y_i)\cap\{x_{n+1}>0\} \,,
\end{equation}
as claimed.

To show that the radii satisfy $r_i=1\pmod2$, we use the varifold convergence of the first variations $\delta V_{\pa\Omega_l}\rightharpoonup\delta V_{\pa\Omega_\infty}$ and Theorem~\ref{white's theorem}. Given a vector field $X\in C^\infty_c(H;\R^{n+1})$ whose support does not intersect $\{x_{n+1}=0\}$, we can use \eqref{supp V_infty = pa Omega_infty} and \eqref{int characterization of Om} to write
\begin{equation}
    \delta V_\infty [X] \;=\; \int_{\pa\Omega_\infty}\Div_\tau X\;\Xi_1\,d\H^n \;=\; \sum_{i\in I} \frac{n}{r_i} \int_{\pa B_{r_i}(y_i)} X\cdot\nu \;\Xi_1\,d\H^n  \,,
\end{equation}
and by using the convergence of the mean curvature \eqref{hyp: conv mc} we can explicitly compute the limit
\begin{align}
    \delta V_\infty[X] 
        &= \lim_{h\to \infty} \int_{M_h} \Div_\tau X \,d\H^n \\
        &= \lim_{h\to \infty} n\int_{M_h} X\cdot\nu \,d\H^n \\
        &= \lim_{h\to \infty} n\int_{\Omega_h} \Div_{n+1} X \,dx \\
        &=n \int_{\Omega_\infty} \Div_{n+1} X \,dx \label{delta V_infty volume} \\
        &=n\int_{\partial\Omega_\infty} X\cdot\nu \,d\H^n  \\ 
        &=\sum_{i\in I} n\int_{\partial B_{r_i}(y_i)} X\cdot\nu \,d\H^n\,,
\end{align}
Hence we may directly conclude that $r_i$ is constantly equal to $\Xi_1$ in each spherical cap $\partial B_{i}$, and the fact that $r_i=\Xi_1=1\pmod2$ follows from Theorem \ref{white's theorem}.

\bigskip

\textit{Step five.} In this step we show that for all the radii $r_i$ are all equal to $1$, and to do so we will use a sliding argument. Fix an index $i\in I$. If $r_i=1$, then we are of course done, so suppose for sake of contradiction that $r_i\neq1$. By the previous step, we have that $r_i\geq3$. If $B_{r_i}(y_i)$ does not intersect the plane $\{x_{n+1}=0\}$, the proof that $r_i=1$ follows from obvious modifications of the argument below. Thus, we will assume that $B_{r_i}(y_i)$ intersects the plane $\{x_{n+1}=0\}$ with angle $\theta_i$. We handle the cases $\theta_i\in(0,\pi/2)$ and $\theta_i\in[\pi/2,\pi)$ separately. 

\medskip

{\bf Case I:} $\theta_i\in(0,\pi/2)$. By the Hausdorff convergence of the boundary we have that for any $\e>0$ and any $\eta>0$ there exists $l$ large enough such that $\Sigma_l\subset\{x_{n+1}<\eta/2\}$ and
\begin{equation}
    B_{r_i-\e}(y_i)\cap \{x_{n+1}>\eta\}\subseteq \Omega_l\cap \{x_{n+1}>\eta\} \,.
\end{equation}
Using that $\theta_i\in(0,\pi/2)$, if we translate $\partial B_{r_i-\e}(y_i)\cap \{x_{n+1}>\eta\}$ vertically then the first touching point with $\partial\Omega_l$ can not happen at $\{x_{n+1}=\eta\}$ for $\eta$ small enough, which implies that the curvatures are ordered $\frac{n}{r_i-\e}\ge H_{\partial\Omega_l}>n-\delta$. Using that $r_i \ge 3$, we get a contradiction by taking $\e$ and $\eta$ small enough.

\medskip

{\bf Case II:} $\theta_i\in[\pi/2,1)$. We first define the height
\begin{equation}\label{h_i}
    h_i := -y_i\cdot e_{n+1} \,,
\end{equation}
which is the distance of the center of the ball $B_1(y_i)$ from the plane $\{x_{n+1}=0\}$ (note the convention $h_i\geq0$).
We will now let $\e>0$ be such that 
\begin{equation}\label{sliding eps}
    \e \;<\; \min\left\{\frac{1}{2r_i}\,,\; \frac{r_i-h_i}{r_i+h_i} \right\}
\end{equation}
By the local Hausdorff convergence of the boundary (setting $\eta=\e$ for simplicity), for $l$ large enough
\begin{equation}
   \emptyset\ne M_l \cap B_{r_i+\e}(y_i)\cap \{x_{n+1}>\e\} \;\subseteq\; B_{r_i-\e}(y_i)^c\cap \{x_{n+1}>\e\}\,.
\end{equation}
If necessary, we will take $l$ larger so that $\Sigma_l\subset\{x_{n+1}<\e\}$. For $t\in(0,1+\e)$ the mapping
\begin{equation}\label{sliding}
    t \;\mapsto\; B_{r_i-1}(y_i+te_{n+1})\cap\{x_{n+1}>\e\}
\end{equation}
corresponds to vertically sliding the ball $B_{r_i-1}(y_i)$ until it touches the boundary $\partial B_{r_i+\e}(y_i)$, which serves as an upper for the furthest one would need to translate in order to touch $M_l$. If the sliding ball \eqref{sliding} does not intersect the boundary of the disk $B_{r_i-\e}(y_i) \cap \{x_{n+1}=\e\}$ for any $t\in(0,1+\e)$, then the first point of contact with $M_l$ will be tangential, and the same argument will hold as in {\bf Case I}, thereby allowing us to conclude that $r_i=1$. Thus, we are reduced to showing that
\begin{equation}\label{sliding min dist}
    \min_{t\in[0,1+\e]}\dist\Big(\partial B_{r_i-1}(y_i+te_{n+1})\cap \{x_{n+1}=\e\}\,,\, \partial B_{r_i-\e}(y_i)\cap \{x_{n+1}=\e\}\Big) \;>\;0 \,.
\end{equation}
We note that both of the sets considered in \eqref{sliding min dist} are disks with the same center, hence to check the desired property we need to show that for $t\in[0,1+\e]$
$$
{\rm radius} (\partial B_{r_i-1}(y_i+te_{n+1})\cap \{x_{n+1}=\e\})< {\rm radius} (\partial B_{r_i-\e}(y_i)\cap \{x_{n+1}=\e\}).
$$
Computationally, we break the proof into two subcases, {\bf I} when the minimum is achieved for some $t\in(0,1+\e)$, and {\bf II} when it is achieved at $t=1+\e$. To alleviate the notation, without loss of generality we consider the case that $y_i$ is on the $x_{n+1}$-axis.


\medskip

{\bf Subcase I:} There exists a time $t_0\in(0,1+\e)$ such that
\begin{equation}\label{caseI}
{\rm radius} (\partial B_{r_i-1}((-h_i+t_0)e_{n+1})\cap \{x_{n+1}=\e\})=r_i-1,    
\end{equation}
which implies that we have the inequality
\begin{equation}\label{eq:li}
h_i=t_0-\e< 1,
\end{equation}
as $t_0\in(0,1+\e)$.
Using the Pythagorean theorem, we see that the desired inequality \eqref{sliding min dist} is equivalent to the claim 
\begin{equation}\label{trig2}
    {\rm radius}\big( B_{r_i-\e}(-(h_i+\e) e_{n+1})\cap \{x_{n+1}=0\} \big) =\sqrt{(r_i-\e)^2-(h_i+\e)^2} \geq r_i-1 \,,
\end{equation}
which directly follows from the fact that $r_i\ge 3$, $h_i<1$ by \eqref{eq:li} and $\e<\frac{1}{2r_i}$ by \eqref{sliding eps}.



\medskip

{\bf Subcase II:} Assuming \eqref{caseI} does not hold,  the desired \eqref{sliding min dist} follows from the inequality
\begin{equation}\label{trig4}
{\rm radius }\Big( B_{r_i-\e}\big(-(h_i+\e) e_{n+1}\big)\cap \{x_{n+1}=0\} \Big)\geq {\rm radius }\Big(B_{r_i-1}\big((-h_i+1)e_{n+1}\big)\{x_{n+1}=0\} \Big)\,.  
\end{equation}
Using the Pythagorean theorem and squaring both sides, this is equivalent to 
$$
(r_i-\e)^2-(h_i+\e)^2\ge (r_i-1)^2-(-h_i+1)^2 \,
$$
which follows from the choice of $\e<\frac{r_i-h_i}{r_i+h_i}$ in \eqref{sliding eps}.


\bigskip

\textit{Step six.} In this step we will show that the angles with which the components of $\Omega_\infty$ intersect the plane $\{x_{n+1}=0\}$ must all be equal to $\theta_0$. More specifically, if we let $I'\subseteq I$ denote the subset of indices such that $B_1(y_i)\cap \{x_{n+1}=0\} \neq\emptyset$, and if for $i\in I'$ we let $\theta_i$ denote the angle with which $B_1(y_i)$ intersects $\{x_{n+1}=0\}$, then we claim that  $\theta_i=\theta_0$ for all $i\in I'$. As a byproduct of the proof, we will also show that $P(\Omega_l)\to P(\Omega_\infty)$.

We consider a test vector field $X\in C^\infty_c(\R^{n+1})$ such that $X\cdot e_{n+1}=0$ on $\{x_{n+1}=0\}$, for which it holds that
\begin{equation}
    \delta V_l[X]=\int_{M_l} H_l X\cdot \nu_{\Omega_l} \,d\H^n+\int_{\pa \Sigma_l}\Big(\sin(\theta_l)X\cdot\nu_{K_l}-\cos(\theta_l) X\cdot\nu_{\Sigma_l}\Big)\,d\H^{n-1} \,.
\end{equation}
Using the varifold convergence $V_l\weak V_\infty$, the convergence of both the mean curvature \eqref{hyp: conv mc} and angle \eqref{hyp: conv angle}, the convergence of $\nu_{K_l}\to e_{n+1}$ by \eqref{containers conv to half space}, and the uniform bound on $\H^{n-1}(\pa\Sigma_l)$ from Lemma~\ref{boundpasigma} , we can apply the divergence theorem twice to obtain that
\begin{equation}
    \delta V_\infty[X]=\lim_{l\to \infty}\delta V_l[X]=\lim_{l\to\infty} \left(n\int_{\Omega_l}\Div_{n+1} X\,dx -\cos(\theta_0)\int_{\Sigma_l} \Div_n X\;d\H^n \right) \,,
\end{equation}
which in turn implies by \eqref{conv in volume} and \eqref{conv in wetted} that
\begin{equation}
    \delta V_\infty[X]=n\int_{\Omega_\infty}\Div_{n+1} X\,dx -\cos(\theta_0)\int_{\Sigma_\infty} \Div_n X\,d\H^n \,.
\end{equation}
Since the boundary of $\Omega_\infty=\bigcup_{i\in I} B_1(y_i)\cap \{x_{n+1}>0\}$ has constant mean curvature equal to $n$, we can apply the divergence theorem again to obtain that
\begin{align}
    \delta V_\infty[X] 
    \;&=\; \sum_{i\in I}\int_{\partial B_1(y_i)\cap H}\Div_{\tau} X\,d\H^{n} \\
    \;&\hspace{.7cm}+\; \sum_{i\in I'}\cos(\theta_i)\int_{B_1(y_i)\cap \{x_{n+1=0}\}}\Div_n X\,d\H^n -\cos(\theta_0)\int_{\Sigma_\infty} \Div_n X\,d\H^n \,,\label{first var rep 1} \,.
\end{align}
At the same time, we claim that we can write $\delta V_\infty$ as
\begin{equation}\label{first var rep 2}
    \delta V_\infty[X] \;=\; \sum_{i\in I} \int_{\partial B_1(y_i)\cap H} \Div_\tau X d\H^n \,+\,\int_{\Omega_\infty\cap\{x_{n+1}=0\}\setminus \Sigma_\infty}\Div_n X\,\Xi_1\, d\H^n \,+\,\int_S \Div_n X\,\Xi_0\, d\H^n \,,
\end{equation}
for $S$ an $n$-rectifiable subset of the plane $\{x_{n+1}=0\}$ disjoint from $\Omega_\infty\cap\{x_{n+1}=0\}\setminus \Sigma_\infty$, and $\Xi_0$ and $\Xi_1$ integer-valued Borel functions such that $\Xi_0 = 0 \pmod{2}$ and $\Xi_1=1\pmod{2}$. Indeed, for $W_\infty$ defined as in Theorem~\ref{white's theorem}, we can use \eqref{white's thm representation} and {\it Step five} to write
\begin{equation}\label{whites theorem rep redux}
    V_\infty \;=\; W_\infty - d\H^n|_{\Sigma_\infty} \;=\; d\H^n|_{\pa\Omega_\infty\cap \{x_{n+1}>0\}} +\Xi_1 d\H^n|_{\Omega_\infty\cap \{x_{n+1}=0\}\setminus \Sigma_\infty} + \Xi_0d\H^n|_S \,,
\end{equation}
where $S\subseteq\{x_{n+1}=0\}$ is the set from Theorem \ref{white's theorem} enlarged to contain the points in $\Sigma_\infty$ with multiplicity greater than $1$. This proves \eqref{first var rep 2}.

\medskip

We now wish to show that $S$ and $\Omega_\infty\cap \{x_{n+1}=0\}\setminus \Sigma_\infty$ have $\H^n$-measure zero, which implies $P(\Omega_l)\to P(\Omega_\infty)$. 
To this end, comparing \eqref{first var rep 1} and \eqref{first var rep 2}, we may conclude that
\begin{equation}\label{identity}
\sum_{i\in I'} \cos(\theta_i)\chi_{ B_1(y_i)\cap \{x_{n+1}=0\}}-\cos(\theta_0)\chi_{\Sigma_\infty}=\Xi_0\chi_S+\Xi_1\chi_{\Omega_\infty\cap \{x_{n+1}=0\}\setminus \Sigma_\infty}\,.    
\end{equation}
Given $i\in I'$, if the set $B_1(y_i)\cap \{x_{n+1}=0\}\subseteq \Omega_\infty\cap \{x_{n+1}=0\}$ is not contained in $\Sigma_\infty$, then by \eqref{identity} we have that $\Xi_1=\cos(\theta_i)$ on $B_1(y_i)\cap \{x_{n+1}=0\}$. The condition $\Xi_1=1 \pmod{2}$ implies that either $\cos(\theta_i)=1$ or $\theta_i=0$, both of which imply that $\H^n(B_1(y_i)\cap \{x_{n+1}=0\})=0$. Hence, for all $i\in I'$ we must have $B_1(y_i)\cap \{x_{n+1}=0\}\subseteq \Omega_\infty$,  which implies that $\H^n(\Omega_\infty\cap \{x_{n+1}=0\}\setminus \Sigma_\infty)=0$. To show $H^n(S)=0$, we notice that by \eqref{identity} and the previous argument the inclusion $S\subseteq \Sigma_\infty$ holds and, moreover, that on each connected component
$$
\Xi_0\;=\;\cos(\theta_i)-\cos(\theta_0)\;=\;0\pmod{2},
$$
which can only be satisfied if $\theta_i=\theta_0$ for every $i\in I$ and $\Xi_0=0$ vanishes everywhere.

\bigskip

\textit{Step seven.} We claim that $\H^{n-1}|\partial \Sigma_l \weak \H^{n-1}|\partial \Sigma_\infty$. Let $\phi\in C^\infty(\{x_{n+1}=0\})$, and, abusing notation, extend $\phi$ to $\R^{n+1}$ by $\phi(x)=\phi(x_1,\hdots,x_n,0)$, so that $\Div(\phi e_{n+1})=0$. Then, by the convergence in mean curvature and angle it follows that
\begin{align}
    \lim_{l\to\infty} & \int_{M_l}\Div_\tau (\phi e_{n+1}) \\
        \;&=\; \lim_{l\to\infty} \left( \int_{M_l} H_l \phi e_{n+1}\cdot\nu + \int_{\bd(M_l)} \Big(\sin(\theta_h)e_{n+1}\cdot\nu_{K_l} - \cos(\theta_l)e_{n+1}\cdot\nu_{\Sigma_l} \Big) \phi \right) \\
        \;&=\; \lim_{l\to\infty} \left( n\int_{M_l} \phi e_{n+1}\cdot\nu + \sin\theta_0\int_{\bd(M_l)} \phi \right) \\
        \;&=\; \lim_{h\to\infty} \left( n\int_{\Omega_l} \Div(\phi e_{n+1}) - n \int_{\Sigma_l} \phi e_{n+1}\cdot\nu_{K_l} + \sin\theta_0\int_{\bd(M_l)} \phi \right) \\
        \;&=\; \lim_{l\to\infty} \left( - n \int_{\Sigma_l} \phi e_{n+1}\cdot\nu_{K_l} + \sin\theta_0\int_{\bd(M_l)} \phi \right) \,,\label{temp limit}
\end{align}
while at the same time by the varifold convergence \eqref{varifold conv}
\begin{align}
    \lim_{l\to\infty}\int_{M_l}\Div_\tau (\phi e_{n+1}) 
        &= \sum_{i\in I} \int_{\partial B_1(y_i)\cap H} \Div_\tau (\phi e_{n+1})  \\
        &= -n\sum_{i\in I} \int_{B_1(y_i)\cap \{x_{n+1}=0\}} \phi + \sum_{i \in I'} \sin\theta_0 \int_{ B_1(y_i)\cap\{x_{n+1}=0\} } \phi \,. \quad \label{temp limit 2}
\end{align}
By \eqref{containers conv to half space} and the previous step,  first term of \eqref{temp limit} has a well-defined limit given by
\begin{equation}\label{temp limit 3}
    \lim_{l\to\infty} \int_{\Sigma_l} \phi e_{n+1}\cdot\nu_{K_l} \;=\; \int_{\Sigma_\infty} \phi \;=\; \sum_{i\in I'} \int_{B_1(y_i)\cap \{x_{n+1}=0\}} \phi \,,
\end{equation}
and from this it follows that the term $\int_{\bd(M_l)} \phi$ must too have a well-defined limit. Combining \eqref{temp limit}, \eqref{temp limit 2}, and \eqref{temp limit 3}, we find that
\begin{equation}
    \lim_{l\to\infty} \int_{\bd(M_l)} \phi \;=\;  \int_{\bd(M_\infty)} \phi \,,
\end{equation}
as claimed. This concludes the proof of the theorem.
\end{proof}

\section{Convergence of local minimizers}\label{min}
\begin{proof}[Proof of Theorem~\ref{thm:locmin}]
     We argue by contradiction. Assume that $\{\Omega_j\}_{j\in\N}$ is a sequence of smooth VCLMs in the hydrophilic regime with diameter $\e_0$ and volume $|\Omega_j|=m_j\to 0$ satisfying
 \begin{itemize}
    \item $\Omega_j$ is connected,
    \item $P(\Omega_j;K)\le C m_j^{\frac{n}{n+1}}$,
    \item $\Omega_j\subseteq Cm_j^{\frac{1}{n+1}}B_{1}(x_j)$ for some $x_j\in\pa K$,
    \item $\H^n(\Sigma_j)>C^{-1}m_j^{\frac{n}{n+1}}$,
 \end{itemize}
 and such that for any angle $\theta$
 \begin{equation}\label{eq:aa}
    \lim_{j\to0^+}\inf_{x,Q,\theta_*} m^{-1}|\Omega_j\triangle (m^{\frac{1}{n+1}}(x_0+Q[B_{\theta}]))|>0,
\end{equation}
where $Q[B_\theta]$ is a rigid transformation of a ball that intersects the $\{x_{n+1}=0\}$ with angle $\theta$.
 
By the compactness of the container, up to a subsequence we can assume that $x_j\to x_0\in\pa K$. Then by re-centering and re-scaling, we have that (up to a rotation) the sequence
$$
K_j:=\frac{K-x_0}{m_j}\;\to\; \{x_{n+1}>0\}.
$$ 
Moreover, using that the boundary of $K$ is $C^2$, we have that the re-scalings $K_j$ converge to the half-space in the sense of \eqref{containers conv to half space}, which is of course is one of the hypotheses of Theorem~\ref{mainthm}.
 
To analogously re-scale and re-center our sets, we define
$$
    \tilde{E}_j=\frac{E_j-x_0}{m_j} \,,
$$
which now satisfy
 \begin{itemize}
    \item $\tilde{\Omega}_j$ is connected,
    \item $P(\tilde{\Omega}_j;K_j)\le C$,
    \item $\tilde{\Omega}_j\subseteq B_{2C}(0)$,
    \item $\H^n(\tilde{\Sigma}_j)>C^{-1}$.
 \end{itemize}
Moreover, using the local minimality condition (properly re-scaled), we can derive the mean curvature equation
 $$
    H_{\pa\tilde{\Omega}_j}+m^{\frac{1}{n+1}}g=\lambda_j\qquad\mbox{on $\pa \tilde{E}_j$}
 $$
 and Young's law 
 $$
    \nu_{\tilde{\Omega}_j}(y)\cdot \nu_{K^j}(y)=\sigma(m_j y+x_0)\qquad\mbox{on $\tilde{\Sigma}_j\subset\pa K_j$}.
 $$
 Hence, we have that
\begin{equation}\label{deficits}
 \|H_{\pa\tilde{\Omega}_j}-\lambda_j\|_{C^0(\pa\tilde{\Omega}_j)}\le m^{\frac{1}{n+1}}\|g\|_{\infty}\qquad\mbox{and}\qquad\|\theta(x)-\theta_0\|_{C^0(\tilde{\Sigma}_j)}\le C(\theta_0)\|\sigma\|_{Lip(\pa K)}m_j,   
\end{equation}
 where we have used the Lipschitz regularity of the adhesion coefficient in the second inequality.
 
 By Proposition~\ref{lem:deficitconvergence2} and Lemma~\ref{boundpasigma}, we can control the perimeter of the wetted region
 \begin{align}
     \H^n(\pa\tilde{\Sigma}_j)\le C(\theta_0) \lambda_j\left( \H^n(\tilde{\Sigma}_j)+m_j^{\frac{1}{n+1}}\|g\|_{C^0}P(\pa\tilde{\Omega}_j;K_j)\right).
 \end{align}
 Moreover, using that $|\tilde{\Omega}_j|=1$ we can identify the limit of the Lagrange multiplier
 \begin{align}
 &\left|(n+1)\lambda_j+n\cos\theta_0\H^n(\tilde{\Sigma}_j)-n P(\pa\tilde{\Omega}_j;K_j) \right|\\
 &\qquad \le C(\theta_0)d_{\tilde{\Omega}_j}(\e+ \|\theta(x)-\theta_0\|_{C^0(\tilde{\Sigma}_j)})\H^n(\tilde{\Sigma}_j) \lambda_j + m_j^{\frac{1}{n+1}}\|g\|_{C^0}d_{\tilde{\Omega}_j} P(\pa\tilde{\Omega}_j;K_j) \,,
 \end{align}
where $C(\theta_0)$ is the constant from  \eqref{lambdabound} depending on $\theta_0$ and $\e>0$ is the parameter in \textbf{Almost Flatness} which can be taken as small as we like (for $j$ large enough). Using upper and lower bounds on perimeter, we can conclude that, up to a subsequence, there exists $\lambda_0>0$ such that
$$
    \lim_{j\to\infty}\lambda_j=\lambda_0>0 \,,
$$
and hence \eqref{deficits} may be recast as
\begin{equation}\label{deficits2}
 \|H_{\pa\tilde{\Omega}_j}-\lambda_0\|_{C^0(\pa\tilde{\Omega}_j)}\to 0\qquad\mbox{and}\qquad\|\theta(x)-\theta_0\|_{C^0(\tilde{\partial\Sigma}_j)}\to 0.
\end{equation}

The desired convergence will follow from applying Theorem~\ref{mainthm} to the sequence of sets $\{\tilde{\Omega}_j\}$; however, we still need to check the technical condition on the convergence of perimeter. To show this, we will exploit the local minimality property and apply \cite[Theorem 2.9]{dephilippismaggiCAP-ARMA} to show that the perimeter of $\tilde{E}_j$ converges strongly in the limit $j\to\infty$.

First, as the hypotheses of \cite[Theorem 2.9]{dephilippismaggiCAP-ARMA} require the container to be exactly the half space, we need to rectify the boundary. For a small enough radius $r>0$, we can consider the inverse mapping that locally rectifies the boundary of $K$ around $x_0$; more specifically, we take the $C^1$-diffeomorphism $f:\{x_{n+1}>0\}\cap B_{r}\to K\cap B_{r}(x_0)$ that, up to rotation, satisfies
    $$
        \nabla f(x)=\id+O(|x|) \,.
    $$
    We next scale $f$ by the sequence of length scales $m_j$ to obtain the sequence of mappings
    $$
    f_j(x)=m_j^{-\frac{1}{n+1}}\Big(f\big(m_j^{\frac{1}{n+1}}x\big)-x_0\Big)\,,
    $$ 
    which are defined
    \begin{equation}
        f_j\,:\,m_j^{-\frac{1}{n+1}}\Big(\{x_{n+1}>0\}\cap B_{r}(x_0)\Big)\;\to\; m_{j}^{-\frac{1}{n+1}}\Big(K\cap B_{r}-x_0\Big) \,.
    \end{equation}
    Note that for $j$ large enough each $f_j$ is a small deformation of the identity on compact sets, i.e.
    $$
        \nabla f_j(x)=\id+O\big(m_j^{\frac{1}{n+1}}|x|\big) \,.
    $$
    For $j$ large enough, one has the inclusion $\tilde{E_j}\subset m_{j}^{-\frac{1}{n+1}}(K\cap B_{r}-x_0)$, and hence we consider
    $$
    G_j:=f_j^{-1}(\tilde{\Omega}_j),
    $$
    which satisfies $G_j\subseteq\{x_{n+1}>0\}\cap B_{3C}$. 
    Translating the local minimality property of $\Omega_j$ to these new sets, we have that for $j$ large enough the set $G_j$ is a local minimizer of the following anisotropic perimeter functional for variations of diameter less than $\e_0/2$:
    \begin{align}
     &\tilde{\mathcal{F}}[G]=\int_{\pa^* G} |\cof(\nabla f_j(y))\nu_F(y)| d\H^n(y)+\int_{G\cap \partial H} \sigma(f_j(z))J^{\pa H}f_j(z) \;d\H^n(z)\\
     &\qquad\qquad+m_j^{\frac{1}{n+1}}\int_{G} g(f_j(x))Jf_j(x)\;dx \, ,\label{anisotropic}
    \end{align}
    where the $\cof(\nabla f_j(y))$ is the co-factor matrix of $\nabla f_j(y)$. 
    
Next, to apply the compactness result \cite[Theorem 2.9]{dephilippismaggiCAP-ARMA}, we need to transform the wetting energy into a surface integral. Following the idea in \cite[Lemma 6.1]{dephilippismaggiCAP-ARMA}, we consider the vector field
    $$
        T_j(x)=\sigma(f_j(\hat{x}))J^{\pa H}f_j(\hat{x})e_{n+1} \,,
    $$
    where $\hat{x}$ denotes the projection onto the boundary of the half-space $\{x_{n+1}=0\}$, and we notice that for a set of finite perimeter $F\subseteq\dom( f_j)$ we have by the divergence theorem
    $$
        \int_{F\cap \partial H} \sigma(f_j(z))J^{\pa H}f_j(z) \;d\H^n(z)=-\int_{\pa^*F\cap H} T_j(y)\cdot\nu_F(y)\;d\H^n(y)+\int_{F\cap H}\Div\,T_j(x)\;dx \,.
    $$
    Using that $G_j$ is a VCLM of \eqref{anisotropic}, we obtain
    \begin{align*}
        &\int_{\pa^* G_j} |\cof(\nabla f_j(y))\nu_{G_j}(y)|-T_j(y)\cdot\nu_{G_j}(y)\ d\H^n(y)\\
        &\qquad\le \int_{\pa^* G} |\cof(\nabla f_j(y))\nu_{G}(y)|-T_j(y)\cdot\nu_{G}(y)\ d\H^n(y) +(\sup|\Div\, T_j|+m_j^{\frac{1}{n+1}}\sup g)|G_j\triangle G|
    \end{align*}
    for any set $G$ satisfying $\diam(G\triangle G_j)<\e_0/2$. Therefore, $G_j$ is an almost minimizer of the pure anisotropic perimeter with elliptic integrand
    $$
        \Phi_j(y,\nu)=|\cof(\nabla f_j(y))\nu|-T_j(y)\cdot \nu \,.
    $$
    For $j$ large enough $\Phi_j$ are convex, positive one-homogeneous, and satisfy the requirements of \cite[Theorem 2.9]{dephilippismaggiCAP-ARMA}.
    
    Therefore, there exists a limiting set $G_0$ such that $G_j\to G_0$ in $L^1$ and such that on $\{x_{n+1}>0\}$ the norm of the Gauss-Green measures converge $\mu_{G_j}\to\mu_{G_0}$. Using that up to a rigid motion $f_j$ converges in $C^1$ to the identity map on compact sets, we obtain that the original scaled set also converges $\tilde{\Omega}_j \to G_0$ in $L^1$ and $P(\tilde{\Omega}_j;K_j)\to P(G_0;\{x_{n+1}>0\})$ in perimeter. Therefore, $G_0$ is a union of balls of equal radius intersecting with constant angle $\theta_0$. By the local minimality of $\tilde{E}_j$, we obtain that $G_0$ is also a local minima of the Gauss energy $\F_{K,\sigma,0}$, where $K=\{x_{n+1}>0\}$ and the adhesion coefficient $\sigma=\cos\theta_0$ is constant. Lemma~\ref{lem:singleball} below implies that $G_0$ is a single ball that intersects the hyperspace $\{x_{n+1}=0\}$ with angle $\theta_0$, which contradicts \eqref{eq:aa}. This concludes the proof of the theorem.
    \end{proof}
    
    \bigskip
    
    \begin{lemma}\label{lem:singleball}
    Consider the case when $K=\{x_{n+1}>0\}$ and the adhesion coefficient $\sigma$ is constant. Then, within the class of unions of tangent balls, volume constrained local minimizers of $\mathcal{F}_{K,\sigma,0}$ are a single ball.
    \end{lemma}
    \begin{proof}
    Suppose for sake of contradiction that $G$ is the union of at least two tangent balls. Then, we first note the Euler-Lagrange equations implies that each of them have correct associated contact angle with the plane $\{x_{n+1}=0\}$. 
    
   If the point of tangency of any two balls occurs in the interior $\{x_{n+1}>0\}$, then by classical upper and lower density arguments, see for instance \cite[Lemma 2.8]{dephilippismaggiCAP-ARMA}, yield the desired contradiction. Next, we assume that there exists a point of tangency $x_0$ lying in the plane $\{x_{n+1}=0\}$. Blowing up this configuration at $x_0$, we obtain  (up to a rotation) the wedge
    $$
    \tilde{G}=\left\{x\;:\; 0<x_{n+1}<\frac{|x_{1}|}{|\sigma|} \right\}.
    $$
    Using the local minimality property of $G$, we can deduce that $\tilde{G}$ must also be a local minimizer of the functional
    $$
    \mathcal{F}(E)=P(E;\{x_{n+1}>0\})+\sigma \H^n(E\cap \{x_{n+1}=0\}).
    $$
    We will achieve the desired contradiction by constructing an appropriate competitor.
    
    For a parameter $l>0$ to be determined later, we define the prism set 
    \begin{equation}
        S=\left\{x:|x_i|\le l\;\mbox{for $2\le i\le n$}\right\}\,.
    \end{equation}
    On $S$, we alter the wedge $\tilde{G}$ to obtain a new set $\tilde{F}$ in the following way:
    $$
    \tilde{F}:=\tilde{G}\cap S^c\; \bigcup\;
    \left\{x\;:\; 0<x_{n+1}<r\; \mbox{for $x \in S$ and $|x_1|<|\sigma|r$}  \right\}.
    $$
    We first notice that the wetted regions for $\tilde{G}$ and $\tilde{F}$ are the same; hence to contradict the minimality of $\tilde{G}$ we only need to compute their difference in perimeter. Due to the simple geometry of the competitor, we can compute explicitly
    $$
    \mathcal{F}(\tilde{G})-\mathcal{F}(\tilde{F}) \;=\; 2l^{n-1} r (\sqrt{1+\sigma^2}-|\sigma|)-2|\sigma|(n-1)l^{n-2} r^2
    $$
    By using that $\sqrt{1+\sigma^2}-|\sigma|>0$ for $|\sigma|<1$ and by taking $r\ll l$ and $r$ small enough, we obtain that
    $$
    \mathcal{F}(\tilde{G})-\mathcal{F}(\tilde{F})>0,
    $$
    which contradicts the local minimality property for $\tilde{G}$. This concludes the proof of the lemma.
    \end{proof}

\section{Proof of Theorem~\ref{thm hk sub}}\label{sec:hkideal}

The next lemma contains the core of the Montiel-Ros argument for proving the classical Heintze-Karcher inequality. The argument is detailed for the sake of clarity.

\begin{lemma}\label{lemma montiel ros argument}
  If $N$ is an orientable hypersurface with boundary in $\R^{n+1}$ with finite area and positive mean curvature $H_N$ with respect to the normal vector field $\nu_N$, then
  \begin{equation}
    \label{hk montielros N}
    (n+1)|\psi(\Gamma)|\le\int_N\,\frac{n}{H_N}
  \end{equation}
  where $\psi(x,t)=x-t\,\nu_N(x)$ for every $(x,t)\in N\times\R$,
  \[
  \Gamma=\Big\{(x,t)\in N\times\R:0<t\le\frac1{\k_n(x)}\Big\}\,,
  \]
  and $\k_1\le\k_2\le\cdots\le\k_n$ denote the principal curvatures of $N$.
  
  Moreover, denoting by $\{N_i\}_{i=1}^m$ the connected components of $N$, we see that equality holds in \eqref{hk montielros N} if and only if each $N_i$ is contained in a sphere $\pa B_r(x_i)$ and the cones over $N_i$ with vertex at $x_i$ are mutually disjoint. 
\end{lemma}

\begin{remark}\label{remark montielros}
  {\rm When $N=\pa\Om$ we are in the case considered by Montiel and Ros \cite{montielros}, who complete the proof of \eqref{hk inq} by showing that $\Om\subseteq\psi(\Gamma)$. Indeed, pick $y\in\Om$ and let $x_0\in\pa\Om$ such that $|x_0-y|=\dist(x_0,\pa\Om)$. Then $\nu_\Om(x_0)=(x_0-y)/|x_0-y|$ and thus $\k_n(x_0)\ge|x_0-y|^{-1}$. In particular, $y=\psi(x_0,|x_0-y|)$ for $|x_0-y|\le1/\k_n(x)$, and thus $\Om\subseteq\psi(\Gamma)$.}
\end{remark}

\begin{proof}[Proof of Lemma \ref{lemma montiel ros argument}]
  The tangential Jacobian of $\psi$ along $N\times\R\subseteq\R^{n+2}$ is given by
  \[
  J\psi(x,t)=\prod_{i=1}^n\big(1-t\,\k_i(x)\big)\qquad\forall (x,t)\in N\times\R\,.
  \]
  If $t\le1/\k_n(x)$, then each factor in this formula is non-negative, and thus one can apply the arithmetic-geometric mean inequality to infer that
  \[
  J\psi(x,t)\le\Big(\frac1n\sum_{i=1}^n(1-t\,\k_i(x))\Big)^n=\Big(1-t\,\frac{H_N(x)}n\Big)^n\,,\qquad\forall (x,t)\in\Gamma\,.
  \] 
  By the area formula we thus get
  \[
  |\psi(\Gamma)|\le\int_{\psi(\Gamma)}\H^0(\psi^{-1}(y))\,dy=\int_{\Gamma}J\psi\,d\H^{n+1}\le \int_{\Gamma}\Big(1-t\,\frac{H_N}n\Big)^n\,d\H^{n+1}
  \]
  where $\H^{n+1}$ is the $(n+1)$-dimensional measure in $\R^{n+2}$. By the coarea formula and since $H_N/n\le \k_n$,
  \begin{eqnarray*}
  \int_{\Gamma}\Big(1-t\,\frac{H_N}n\Big)^n\,d\H^{n+1}
  &=&\int_Nd\H^n\int_0^{1/\k_n(x)}\Big(1-t\,\frac{H_N}n\Big)^n
  \\
  &\le&\int_Nd\H^n\int_0^{n/H_N}\Big(1-t\,\frac{H_N}n\Big)^n
  \\
  &=&\frac1{n+1}\int_N\,\frac{n}{H_N}\,.  
  \end{eqnarray*}
  This proves \eqref{hk montielros N}. If equality holds, then (i) $N$ is umbilical, in the sense that $\k_i(x)=H_N(x)/n$ for every $x\in M$ and $i=1,...,n$; (ii) $\psi$ is injective on $\Gamma$. By (i), each connected component $N_i$ of $N$ is either contained in a sphere or in a hyperplane, the latter case being actually ruled out by $H_N>0$. Denote by $x_i$ the center of the sphere containing $N_i$, and by $K_i$ the cone with vertex $x_i$ spanned by $N_i$. By (ii), the cones $K_i$ are mutually disjoint.
\end{proof}
Now we are ready to prove Theorem~\ref{thm hk sub}.
\begin{proof}
  [Proof of Theorem \ref{thm hk sub}] We apply Lemma \ref{lemma montiel ros argument} with $N=M$ and find
  \begin{equation}
    \label{pt1 1}
  (n+1)|\psi(\Gamma)|\le\int_M\frac{n}{H_M}
  \end{equation}
  where $\psi(x,t)=x-t\,\nu_\Om(x)$, $\Gamma=\{(x,t):x\in M\,,0<t\le1/\k_n(x)\}$ and $\k_1\le\cdots\le\k_n$ are the principal curvatures of $M$. We notice that
  \begin{equation}
    \label{pt1 2}
      A\subseteq\psi(\Gamma)
  \end{equation}
  if $A$ denotes the set of points $y\in\R^{n+1}$ such that
  \begin{equation}
    \label{pt1 def A}
      \mbox{$\dist(y,M)=|x-y|$ for some $x\in M\cap H$ with $\frac{x-y}{|x-y|}=\nu_\Om(x)$}\,.
  \end{equation}
  Indeed, this is the same argument outlined in Remark \ref{remark montielros}: if $y\in A$ then the sphere $\pa B_{|x-y|}(y)$ is touching $M$ at $x$ from the inside of $\Om$. Hence
  \[
  \frac1{|x-y|}\ge\k_n(x)
  \]
  which in turn implies $y=\psi(x,|x-y|)$ for $(x,|x-y|)\in\Gamma$.
  
  \medskip
  
  \noindent {\it Step one}: We show that if $y\in\R^{n+1}$ is such that $\dist(y,M)=|x_0-y|$ for some $x_0\in\bd(M)$, then
  \begin{equation}
    \label{pt1 3}
      \frac{x_0-y}{|x_0-y|}=\cos\a\,\nu_H+\sin\a\,\nu_\Si(x_0)\qquad\sin(\theta(x_0)-\a)\le0\,,
  \end{equation}
  
  Indeed, by differentiating $x\mapsto|x-y|^2$ at $x_0$ along a direction $v\in T_{x_0}(\pa\Si)$ we find that
  \[
  (x_0-y)\cdot v=0\qquad \forall v\in T_{x_0}(\pa\Si)=\nu_\Si(x_0)^\perp\cap\nu_H^\perp\,.
  \]
  This implies that $x_0-y$ lies in the plane generated by $\nu_H$ and $\nu_\Sigma(x_0)$, thus the first relation in \eqref{pt1 3}. Next, by differentiating $x\mapsto|x-y|^2$ at $x_0$ in the direction $-\nu_{\bd(M)}^M(x_0)$ (which points inwards $M$) we find that
  \[
  (x_0-y)\cdot\big(-\nu_{\bd(M)}^M(x_0)\big)\ge0
  \]
  which thanks to \eqref{nuMbdM} and the first relation in \eqref{pt1 3} gives
  \[
  0\le -\cos\a\sin\theta(x_0)+\sin\a\cos(\theta(x_0))=\sin(\a-\theta(x_0))\,.
  \]
  This proves \eqref{pt1 3}.

  \medskip
  
  \noindent {\it Step two}: We assume to be in the hydrophobic case where $0<\theta_{min}\le\theta_{max}\le\pi/2$. Setting $\Lambda$ as in \eqref{lambda hydrophobic}, it will suffice to prove
  \begin{equation}
    \label{pt1 inclusion hydrophobic}
  \Om\setminus\ov{\Lambda}\subseteq A\,.
  \end{equation}
  Indeed, by combining \eqref{pt1 1} and \eqref{pt1 2} with \eqref{pt1 inclusion hydrophobic} we find
  \[
  \frac1{n+1}\int_M\frac{n}{H_M}\ge|\psi(\Gamma)|\ge|A|\ge|\Om\setminus\Lambda|\ge|\Om|-|\Lambda|\,,
  \]
  that is the hydrophobic case of \eqref{hk inq sub proof}.
  
  We now prove \eqref{pt1 inclusion hydrophobic}. Let $y\in\Om\setminus\ov{\Lambda}$, and let $x_0\in M$ be such that $|x_0-y|=\dist(y,M)$. If $x_0\in M\cap H$, then according to \eqref{pt1 def A} we immediately find $y\in A$, as desired. We conclude by showing that if $x_0\in\bd(M)$, then $y\in\ov{\Lambda}$. Indeed, if this is the case, then \eqref{pt1 3} holds. Since $y_{n+1}>0$, in \eqref{pt1 3} we have $\a\in(-\pi/2,\pi/2)$. By  $\theta(x_0)\in(0,\pi/2]$, we have $\theta(x_0)-\a\in(-\pi/2,\pi)$, so that $\sin(\theta(x_0)-\a)\le0$ implies
  \[
  -\frac{\pi}2< \theta(x_0)-\a\le0\,.
  \]
  Summarizing,
  \begin{equation}
    \label{pt1 alfa larger thetamin}
      0<\theta_{min}\le\a<\frac{\pi}2\,.
  \end{equation}
  Denoting $y=(z,t)$ and multiplying \eqref{pt1 3} by $x_0-y$, we obtain
  $$
    0\le |x_0-y|\le -\cos \a t+\sin \a |z-x_0|
  $$
  which implies that
  $$
   0\le t\le\tan \a \;  d(P_{\pa H}y,\pa\Sigma),
  $$
  which is the definition of $y\in\Lambda$.

  \medskip
  
  \noindent {\it Step three}: We address the hydrophilic case $\pi/2\le\theta_{min}\le\theta_{max}<\pi$ by showing that
  \begin{equation}
    \label{pt1 inclusion hydrophilic}
  \Om\cup\Lambda\subseteq A\,,
  \end{equation}
  where now $\Lambda$ is as in \eqref{lambda hydrophilic}. As before, we just need to exclude the existence of $y\in\Om\cup\Lambda$ such that $\dist(y,M)=|x_0-y|$ for some $x_0\in\bd(M)$. Looking at \eqref{pt1 3}, if $y\in\Om$, then $y_{n+1}>0$, which implies $\a<\pi/2$. If $y\in\Om$, but $\p y\notin\Si$, then the closest point has to happen at $M$. Therefore, we can assume that $\p y\in\Si$ and thus $\a>0$. Since $\a\in(0,\pi/2)$, and $\theta(x_0)\in[\pi/2,\pi)$, we conclude that $\theta(x_0)-\a\in(0,\pi)$, $\sin(\theta(x_0)-\a)>0$, and thus find a contradiction with the second condition in \eqref{pt1 3}. This proves that $\Om\subseteq A$. Now let us consider the case of $y\in\Lambda$. By construction, $0\le\a<\theta_{min}$ so that
  \[
  0<\theta_{min}-\a\le \theta(x_0)-\a\le \theta(x_0)<\pi
  \]
  which, again, contradicts $\sin(\theta(x_0)-\a)\le 0$.  
  
  \medskip
  
  \noindent {\it Step four}: Let us now assume that $\Om$ satisfies \eqref{hk inq sub proof} as an equality. This means that \eqref{pt1 1} holds as an equality. By Lemma \ref{lemma montiel ros argument} and since $\Om$ is connected, $M\subseteq\pa B_r(x)$ for some $r>0$ and $x\in\R^{n+1}$. The facts that $\bd(M)=M\cap H$ and $\theta(x)\in(0,\pi)$ for every $x\in\bd(M)$ force $\bd(M)$ to be a non-degenerate sphere and $\theta$ to be constant along $\bd(M)$. 
  
  Let us show, conversely, that if $\Om$ is the intersection of an Euclidean ball with $H$ in such a way that $\theta=\theta_{min}=\theta_{max}\in(0,\pi)$, then equality holds in \eqref{hk inq sub proof}. Let us notice that, independently from the value of $\theta$, we always have that
  \begin{equation}
    \label{pt1 equa}
      \Om=H\cap B_r(r\cos(\theta)\,e_{n+1})\,.
  \end{equation}
  Since $M$ is always a spherical cap of radius $r$, we have
  \[
  \int_M\frac{n}{H_M}=r\,\H^n(M)\,.
  \]
  Notice also that $\Lambda$ is always the cone with center at $r\cos(\theta)\,e_{n+1}$ spanned by $B_r(r\cos(\theta)\,e_{n+1})\cap\pa H$. In the hydrophobic case,
  \[
  \Lambda\subseteq\Om\qquad\mbox{with}\qquad \Om\setminus\Lambda=\bigcup_{0<\l<1}\l\,M\,;
  \]
  in the hydrophilic case,
  \[
  \Lambda\cap\Om=\emptyset\qquad\mbox{with}\qquad \Om\cup\Lambda=\bigcup_{0<\l<1}\l\,M\,.
  \] 
  Since in both cases
  \[
  \Big|\bigcup_{0<\l<1}\l\,M\Big|=\int_0^r\,\H^n((\rho/r)M)\,d\rho=
  \int_0^r\,(\rho/r)^n\,d\rho\,\H^n(M)
  =\frac{r\,\H^n(M)}{n+1}
  \] 
  we conclude that \eqref{hk inq sub proof} always holds as an equality.
\end{proof}







\bibliography{ref}
\bibliographystyle{is-alpha}

\end{document}